\documentclass[12pt]{extarticle}
\usepackage{maxout}
 \usepackage{todonotes} 

\title{\bf Maxout Polytopes}
\author{Andrei Balakin, Shelby Cox, Georg Loho and Bernd Sturmfels}
\date{}

\begin{document}

\maketitle

\begin{abstract}
\noindent
Maxout polytopes are defined by
feedforward neural networks with maxout activation function and non-negative weights after the first layer.
We characterize the parameter spaces and extremal f-vectors of maxout polytopes for shallow networks,
and we study the separating hypersurfaces which arise when a layer is added to
the network.
We also show that maxout polytopes are cubical for generic networks without bottlenecks.
\end{abstract}

\section{Introduction}
\label{sec: intro}

Neural networks with ReLU activation are fundamental in machine learning.
They represent piecewise-linear functions, so we can
use polyhedral methods to study these networks. 
There is a considerable body of literature on their complexity and expressivity,
also from a polyhedral perspective~\cite{huchette2025deeplearningmeetspolyhedral}.
The number of linear pieces is studied in articles such as \cite{hanin2019, montufar2014, MRZ, serra2018bounding}. 
Insights on the required depth for exact representation by ReLU neural networks
via the associated polytopes were presented in
   \cite{arora2018understanding,averkov2025on,bakaev2025depthmonotonereluneural, bakaev2025betterneuralnetworkexpressivity,grillo2025depthviabraidarrangement,haase2023lower,
hertrich2023}.
Tropical geometry approaches emphasize the representation of a piecewise-linear function by a 
tropical rational function, which is evaluated by  arithmetic  in the max-plus algebra
\cite{brandenburg2024realtropicalgeometryneural, kordonis2025, maragos2021, zhang18}.

Under certain hypotheses, the tropical rational function is a tropical polynomial.
This means that the function is convex: it is the support function of
a convex polytope.
This is guaranteed if the neural network has only non-negative weights after the first layer, leading to so-called input convex neural networks \cite{amos2017input}.
We restrict to this case and, furthermore, consider 
the maxout activation function instead of ReLU for technical reasons; see \cite[Remark~2.2]{hertich2024extendedformulations}. 

As the weights of the network
vary, interesting classes of polytopes arise.
This perspective was initiated by Valerdi
\cite{valerdi2024minimaldepthneuralnetworks, valerdi2025polytopedepthbounds}.
The correspondence between polytopes and neural networks representing their support function gives rise to a natural notion of depth for polytopes.
 Polytopes of depth $1$ are points.
 Polytopes of depth $m+1$ are obtained by
 taking Minkowski sums of convex hulls of pairs of 
 depth $m$ polytopes.
  In particular, the polytopes of depth $2$ are the {\em zonotopes}, i.e.~Minkowski sums of line segments.
  The purpose of this article is to study interesting polytopes arising in this recursive fashion, generalizing the class of zonotopes: 

  We introduce {\em maxout polytopes} and offer a detailed study of this family of polytopes with controlled depth. 
Such a polytope is obtained from a feedforward
neural network with maxout activation function with two arguments, 
where all weights after the first layer are non-negative.
See Section \ref{sec2} for precise definitions.
Two $3$-dimensional maxout polytopes are depicted in
Figure~\ref{fig:endler}. Each diagram shows the convex hull of
two axis-parallel boxes in $3$-space. 
These polytopes have depth $3$.
A {\em box} is affinely isomorphic
to the standard $3$-cube. 

\begin{figure}[h]
  \begin{minipage}{.5\textwidth}
    \centering
    \includegraphics[width=1\textwidth, trim={6cm 3.4cm 6cm 3.4cm}, clip]{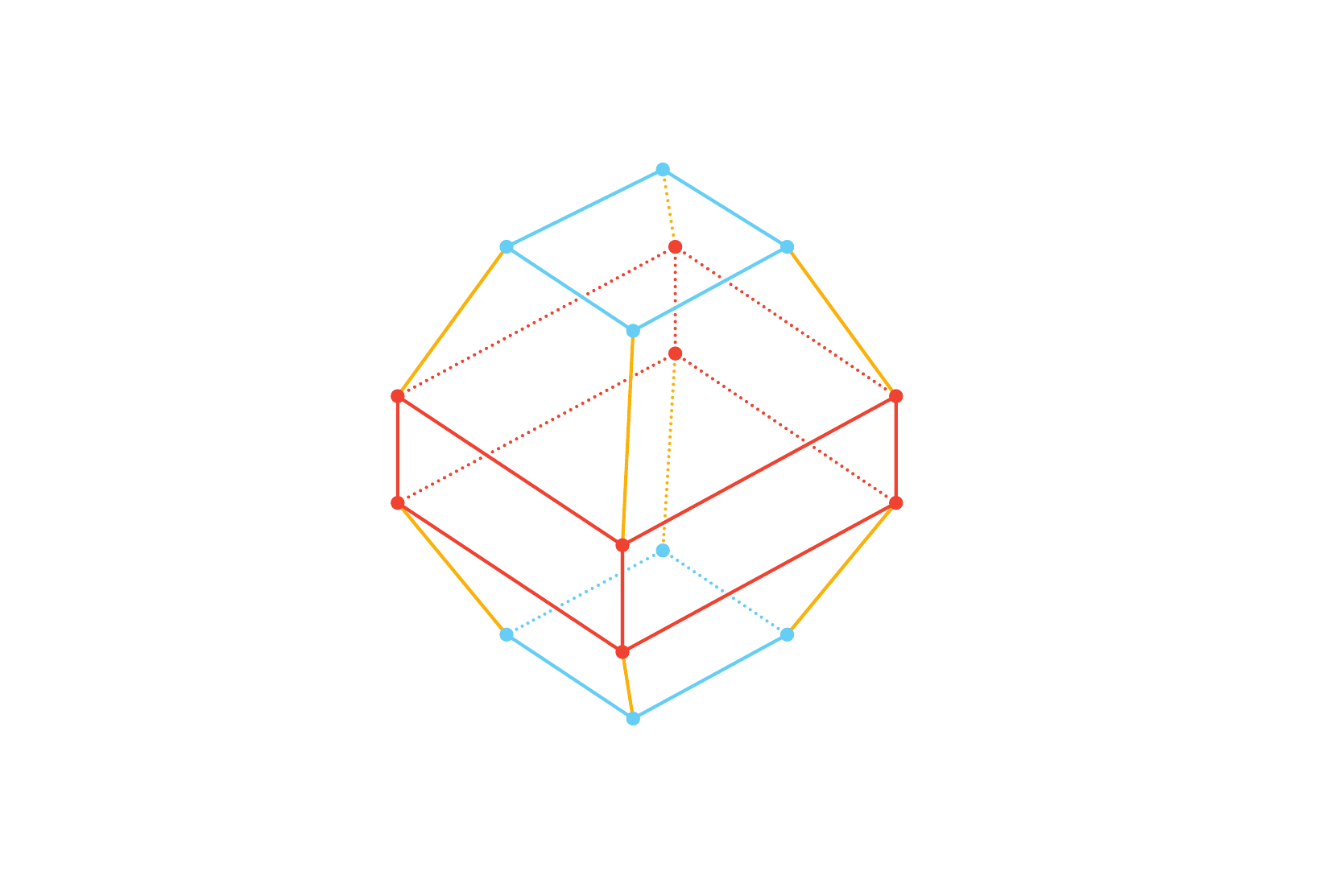}
  \end{minipage}%
  \begin{minipage}{.5\textwidth}
    \centering
    \includegraphics[width=1\textwidth, trim={6cm 3.4cm 6cm 3.4cm}, clip]{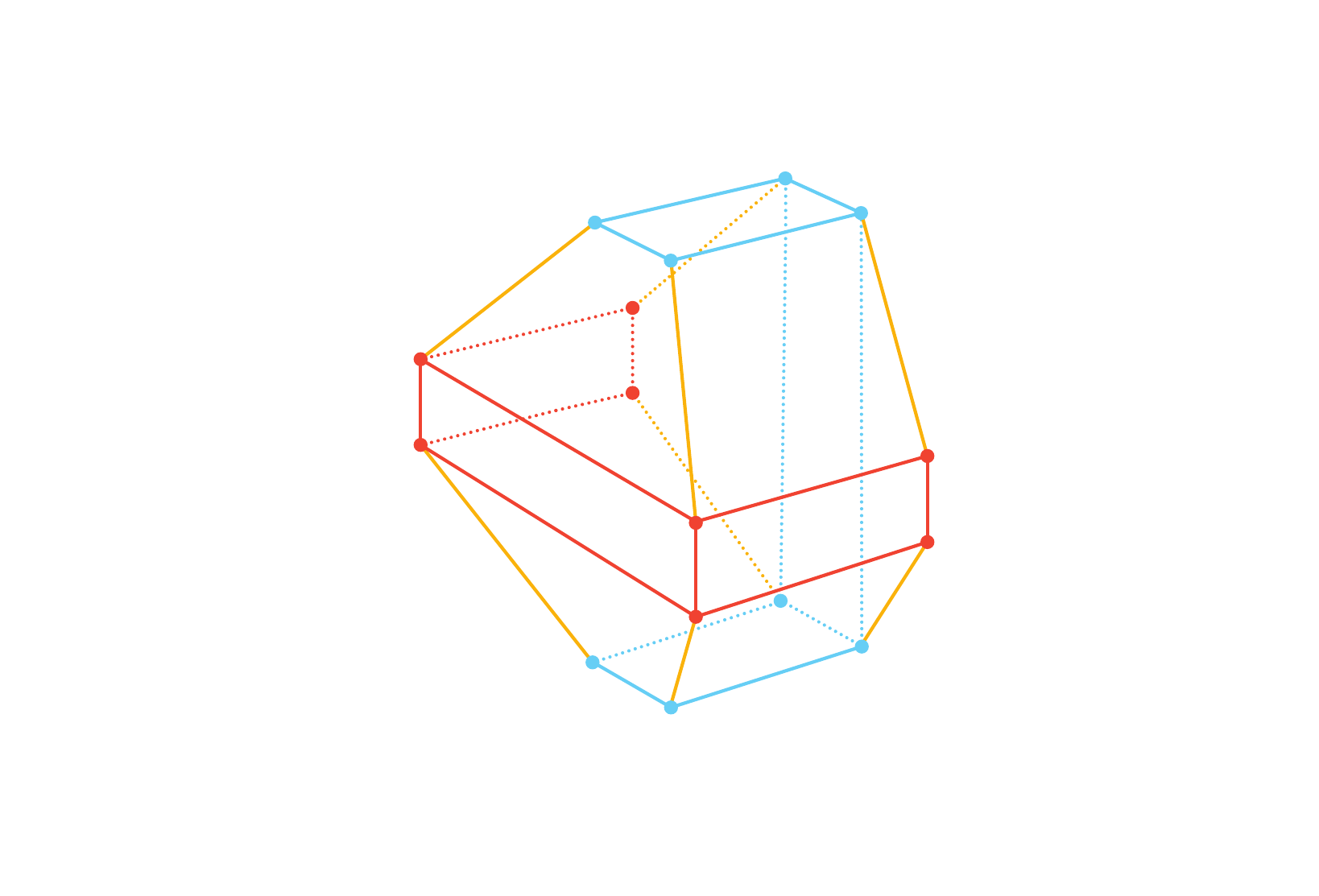}
  \end{minipage}
\caption{Two $3$-dimensional maxout polytopes of type $(3,3,1)$.
See Figure \ref{fig:separating-hypersurface} for their duals.}
\label{fig:endler}
\end{figure}

We use the term {\em boxtopes} for polytopes
like that in Figure \ref{fig:endler}.
The following noteworthy features are easily seen in the diagram.
 Generically, boxtopes are cubical: each facet is~combinatorially a
 cube. These cubical facets come in three flavors: facets of the
 red box, facets of the blue box, and mixed facets. The former
 two are parallelograms, while the latter are trapezoids,
 obtained as the convex hull of a red edge and a blue edge that are parallel.
Our aim is to study such geometric features 
in any dimension and for any network architecture.

The study of cubical polytopes has a long history in geometric combinatorics,
with foundational contributions in the 1990s by Adin  \cite{Adin},
 Babson et al.~\cite{babson1997} and  Blind-Blind \cite{blind1998}.
Our primary reference is the work of Joswig and Ziegler \cite{JZ}.
From the vantage point of these sources, maxout polytopes furnish a
novel construction  of (neighborly) cubical polytopes.

This article is organized as follows.
Section \ref{sec2} offers a self-contained introduction to
maxout polytopes and their underlying neural networks.
Here  we present basic facts and definitions.
In Section \ref{sec3} we focus on boxtopes, which are the simplest
maxout polytopes of depth $\geq 3$, as in Figure~\ref{fig:endler}.
Theorem \ref{thm:maxbox}
identifies boxtopes with extremal face numbers in arbitrary dimensions.
These are the neighborly cubical polytopes which are known from~\cite{JZ}.

An important feature of our theory is the distinction between
maxout polytopes and maxout candidates. The latter
correspond to the bounded depth polytopes of Valerdi
 \cite{valerdi2024minimaldepthneuralnetworks, valerdi2025polytopedepthbounds}.
The maxout polytopes form a proper subclass, consisting of
precisely those polytopes which are realizable by some network  (\ref{eqn:nn}).
This distinction is made precise in Section \ref{sec4}. Example~\ref{ex:blue}
offers an algebraic derivation in dimension two.
We identify the space of maxout polytopes
inside the  ambient space of all candidates.
Both are semialgebraic sets, reminiscent of realization spaces
of oriented matroids, with descriptions by determinantal
equations and inequalities. We compute the dimensions of
these spaces in
Proposition \ref{prop:spaceof} and
Theorem~\ref{thm:fibers}.

In Section \ref{sec:separating-hypersurface} we examine
how the piecewise-linear structure changes per layer.
This rests on a fundamental polytope construction, namely
the convex hull of normally equivalent~polytopes.
We study separating hypersurfaces, which capture the expressivity 
of one layer in a network. Geometrically, these hypersurfaces
 arise by intersecting the
boundaries of two polytopes in general position.
Theorem  \ref{thm: separating fan}
explains the relevance of this construction
for maxout polytopes. Theorem \ref{thm:twocubes}
determines their  topology for the extremal boxtopes in Section~\ref{sec3}.

In Section \ref{sec6} we apply separating hypersurfaces
to  construct extremal maxout polytopes.
We characterize vertex-maximal zonoboxtopes in dimensions
two (Theorem \ref{thm: zonoboxtopes dim two}), we explore dimension
 three (Proposition \ref{prop: upper bound experiments}),
and we propose a general formula in Conjecture~\ref{conj: zonoboxtope max vertices}.

In Section \ref{sec7} we study the question of whether maxout polytopes
are cubical for generic  weights.
 Our main result, Theorem \ref{thm:wide-layers=>cubical}, states that maxout polytopes
 are indeed cubical for generic weights when the network has \emph{no bottlenecks}.
However, interestingly, it fails for networks of type $(3,2,3)$ for instance.
We show this in Proposition \ref{prop:non-cubcial} and Figure \ref{fig:non-cubcial}.
 
\section{From Networks to Polytopes}
\label{sec2}

We review some basics on ReLU neural networks and polytopes; for further details see~\cite{hertrichThesis,zieglerTextbook}.
Consider a function $f$ from $\RR^{m_0}$ to $\RR$ that is represented by a feedforward neural network 
 \begin{equation}
    \label{eqn:nn}
    f\colon
    \RR^{m_0} \,\xrightarrow{f_1}\, \RR^{m_1}   \,\xrightarrow{f_2}\, \RR^{m_2}   \,\xrightarrow{f_3}
     \,\,\cdots\,\, \xrightarrow{f_\ell}\, \RR^{m_\ell} \,\xrightarrow{g}\, \RR \enspace . 
\end{equation}
Each  function $f_i$ is the composition of a linear map followed by a non-linear coordinatewise \emph{activation function}. 
Here $i$ ranges over $[\ell] = \{1,2,\ldots,\ell\}$.
In this article we use the \emph{maxout} activation function with two arguments.
This means that $f_i$ and $g$ can be written as
\begin{gather}
  \label{eqn:maxout-functions}
  f_i \colon \RR^{m_{i-1}} \to \RR^{m_i} , \quad \; i \in [\ell]; \qquad \qquad g : \RR^{m_\ell} \to \RR, \\
  \bfx \mapsto \max(A_i\bfx, B_i\bfx). \hphantom{\quad \; i \in [\ell];} \qquad \quad \qquad x \mapsto Cx. \nonumber
\end{gather}
Here $A_i, B_i$ are matrices of size $m_i \times m_{i-1}$, $\max$ is the coordinatewise maximum, and $C$ is a row vector
of length $m_\ell$.
The entries of $A_i, B_i$ and $C$ are referred to as {\em network weights}.

The network is organized into layers.
The \emph{input layer} takes an input vector $\bfx$ of dimension $m_0$.
The network computes the sequence of intermediate values in \emph{hidden layers} as vectors of size $m_1, m_2, \ldots, m_\ell$,
 and it finally provides the \emph{output} $f(\bfx)$, which is a real number.

The sequence $(m_0, m_1, \ldots, m_\ell)$ from~(\ref{eqn:nn}) is the \emph{type} of the network, and the number $\ell+1$ is its \emph{depth}. 
We require the weights to be nonnegative after the first hidden layer,
i.e.,
\begin{gather}
    A_1, B_1 \in \RR^{m_1 \times m_0}, \quad
    A_i, B_i \in \RR_{\geq 0}^{m_i \times m_{i - 1}} \,\,\text{ for } i = 2, 3, \ldots, \ell,
    \quad
    C \in \RR_{\geq 0}^{1 \times m_\ell}.
    \label{eqn:nonnegative-weights}
\end{gather}
Networks with such weight constraints are called \emph{input-convex maxout (neural) networks} \cite{amos2017input}.

The function~$f$ computed by a network~(\ref{eqn:nn}) of the form~(\ref{eqn:maxout-functions}) is \emph{continuous piecewise linear} \emph{(CPWL)} and \emph{positively homogeneous};
the latter means that $f(\lambda \bfx) = \lambda f(\bfx)$ for any $\lambda \in \RR_{\geq 0}$.
When the weight restrictions~(\ref{eqn:nonnegative-weights}) are satisfied, the function~$f$ becomes convex.
This earns the network the name input-convex.
Input-convex maxout networks give rise to polytopes.
The construction is described below. For this, 
the positivity hypothesis in (\ref{eqn:nonnegative-weights}) is essential.

A \emph{polytope} $P$ is the convex hull of finitely many points in a real vector space~$\RR^{m_0}$.
Polytopes are bounded and closed.
Hence, any linear functional $\,  \RR^{m_0} \to \RR, \, \bfp \mapsto \bfx^\top \bfp$,
for $\bfx \in \RR^{m_0}$, when restricted to $P$, achieves its maximal value on some subset of~$P$.
The subsets which maximize linear functionals on the polytope are called the \emph{faces} of $P$.
The set of faces ordered by inclusion forms a lattice, which defines the \emph{combinatorial type} of the polytope.
  Two polytopes are \emph{combinatorially equivalent} if their face lattices are isomorphic.

The function $f_P \colon \RR^{m_0} \to \RR$ defined by 
$\bfx \mapsto f_P(\bfx) := \max_{\bfp \in P} \bfx^\top \bfp$ is the \emph{support function} of the polytope~$P$.
Its epigraph $\epi f_P := \{(\bfx, y)\in \RR^{d + 1} \,\mid\, y \geq f_P(\bfx)\}$ is a convex cone in $\RR^{d + 1}$.
  The images of the cones in the boundary complex $\partial \epi f_P$ under the projection $\pi \colon \RR^{d + 1} \to \RR^d$ eliminating the last coordinate form the \emph{normal fan} of $P$ in $\RR^d$.
  The linear pieces of the support function $f_P$ are exactly the cones of the normal fan of $P$.
  Two polytopes are \emph{normally equivalent} if they have the same normal fan.
  A polytope $P$ is a \emph{deformation} of a polytope $Q$ if the normal fan of $Q$ is a \emph{refinement} of the normal fan of $P$. That is, if each cone of the normal fan of $Q$ is contained in a cone of the normal fan of $P$. 

Support functions of polytopes are convex, positively homogeneous and CPWL.
These functions can be regarded as polynomials (with real exponents)
in the max-plus (tropical) algebra
\cite{brandenburg2024realtropicalgeometryneural, zhang18}.
They are expressions of the form $\bigoplus_{\bfv \in V} a_{\bfv} \odot x^{\odot \bfv}$, where the addition $\oplus$ is $\max$, the multiplication $\odot$ is $+$, and $V$ is a finite subset of $\RR^{m_0}$ and the $a_{\bfv}$ are real numbers. 

Conversely,
 any convex, positively homogeneous CPWL function $f\colon \RR^{m_0} \to \RR$ is a support function of some polytope
 in $\RR^{m_0}$.
  This polytope is denoted $\Newt(f)$ and called the \emph{Newton polytope} of $f$.
The Newton polytope can be constructed as follows.
Any function $f$ with the properties as above can be represented in the form
$f(\bfx) = \max(\bfx^\top \bfp_1, \bfx^\top \bfp_2, \ldots, \bfx^\top \bfp_n)$
for some points $\bfp_1, \bfp_2, \ldots, \bfp_n \in \RR^{m_0}$.
This implies that $\Newt(f) = \conv(\bfp_1,  \bfp_2, \ldots, \bfp_n)$.

The bijection between polytopes and convex, positively homogeneous CPWL functions respects some 
natural operations on these objects. The following three properties hold:
\begin{itemize}
    \item $\Newt(\lambda f) = \lambda \Newt(f)$ for any $\lambda \in \RR_{\geq 0}$; scaling of the function corresponds to \emph{dilation} of the Newton polytope.
    \item $\Newt(f + g) = \Newt(f) + \Newt(g)$; pointwise addition of functions corresponds to the \emph{Minkowski sum} of two polytopes, defined by
    $\,P + Q \,:=\, \{\bfp + \mathbf{q} \,\mid\, \bfp \in P, \mathbf{q} \in Q\}$.
    \item $\Newt(\max(f, g)) = \conv(\Newt(f) \cup \Newt(g))$.
\end{itemize}
This correspondence of basic operations for functions and polytopes allows one to translate between constructions of neural networks and polytopes. 
Linear forms correspond to weighted Minkowski sums, while convex hulls of polytopes encode maxout activations.

A polytope $P$ is a \emph{maxout polytope of type} $(m_0, m_1, \ldots, m_\ell)$ if it is the Newton polytope of a function computed by some input-convex maxout network of the same type.
Equivalently, $P$ arises from the  layer-by-layer construction in (\ref{eq:recursive-polytopes}).
The input layer consists of the standard coordinate points $\bfe_i$ in $\RR^{m_0}$.
The $k$-th layer consists of $m_k$ polytopes, each of which is a convex hull of two weighted Minkowski sums of the polytopes from the previous layer.

The final polytope~$P$ is a weighted Minkowski sum of the polytopes in the $\ell$-th layer:
\begin{gather} 
  P_{0, i} = \{\bfe_i\} \subset \RR^{m_0}, \,\,i \in [m_0];
  \nonumber\\
  P_{k, i} = \conv\left(
    \sum_{j = 1}^{m_{k - 1}} a_{k, ij} P_{k - 1, j} \,\cup\, \sum_{j = 1}^{m_{k - 1}} b_{k, ij} P_{k - 1, j}
  \right),\; k \in [\ell],\; i \in [m_k];
  \;
  P = \sum_{j = 1}^{m_\ell} c_j P_{\ell, j}.
  \label{eq:recursive-polytopes}
\end{gather}
Here we use the network weights from before, namely $\,a_{1, ij}, b_{1, ij} \in \RR$ for $ (i, j) \in [m_1] \times [m_0]; \; a_{k, ij}, b_{k, ij} \in \RR_{\geq 0}$
for $ k = 2, 3, \ldots, m_\ell$ and $ (i, j) \in [m_k] \times [m_{k - 1}]; \; c_j \in \RR_{\geq 0}$ for $ j\in [m_\ell]$.

The compatibility of the operations above ensures that
  the support function of the polytope $P_{k,i}$ is indeed the function computed by the $i$th node of the $k$th layer.
  We note that every polytope can be constructed in this way. 
  To do this, we  generate all vertices in the first layer, and then iterate the binary $\conv$ operation until all points are gathered in one convex hull.
Translating this construction back to the level of networks, one sees that input-convex maxout networks are capable of computing the support function of any polytope.

\begin{proposition}
  The class of functions computable by input-convex maxout
  neural networks coincides with the class of support functions of polytopes.
  They both consist of all convex, positively homogeneous CPWL functions.
\end{proposition}

Although any polytope can be viewed as a maxout polytope, providing restrictions on the type
$(m_0, m_1, \ldots, m_\ell)$  defines interesting classes of polytopes.
For example, restricting the network depth to $\ell = 2$ yields the class of \emph{zonotopes}.
A zonotope is the Minkowski sum of a finite set of line segments.
These are the \emph{generators} of the zonotope.
Its  \emph{zones} are the collections of its proper faces having a given generator as a Minkowski summand.
The maxout polytopes of type $(d, n)$ are exactly the zonotopes with at most $n$ zones in $\RR^d$.

Maxout polytopes of type $(d, d, 1)$ are called $d$-boxtopes. A boxtope can be represented~as
\begin{equation}
  \conv\left(
    \sum_{i = 1}^d a_i I_i \, \cup \, \sum_{i = 1}^d b_i I_i
  \right) \,\subset \, \RR^d,
  \label{eqn:boxtope-def}
\end{equation}
where $I_i \subset \RR^d$ is a line segment and $a_i, b_i \in \RR_{\geq 0}$ for $i\in [d]$.
Two $3$-boxtopes are shown in Figure~\ref{fig:endler}.
More generally, maxout polytopes of type $(d, n, 1)$ are called $(d, n)$-\emph{zonoboxtopes}.

Every boxtope is a convex hull of two parallelepipeds with axis-parallel edges.
It equals
\begin{equation}
  \conv\left(
    B^{(a)} \, \cup \, B^{(b)}
  \right) \,\subset\, \RR^d,
  \label{eqn:boxtope-candidate-def}
\end{equation}
where 
$B^{(a)} = \sum_{i = 1}^d I^{(a)}_i, \,B^{(b)} = \sum_{i = 1}^d I^{(b)}_i$, and 
$I^{(a)}_i \parallel I^{(b)}_i$ are parallel line segments in $\RR^d$ for $i \in [d]$.
 Line segments are allowed to have length zero.
However, the formula~(\ref{eqn:boxtope-candidate-def}) defines a
 strictly broader class of polytopes than~(\ref{eqn:boxtope-def}).
See Example \ref{ex:blue}.
The subtle difference lies in the possibility to take translated line segments. 
Translation is allowed in (\ref{eqn:boxtope-candidate-def}) but not in~(\ref{eqn:boxtope-def}).

To capture this distinction, we introduce classes of polytopes to be called `\emph{candidates}'.
A $(d, n)$-\emph{zonoboxtope candidate} is a convex hull of two zonotopes in $\RR^d$, each with $n$ generators,
where  corresponding generators are parallel.
A $d$-\emph{boxtope candidate} is a $(d, d)$-zonoboxtope candidate. 
Finally, a \emph{$(d,n,m)$-maxout candidate} is the Minkowski sum of $m$ scaled $(d, n)$-zonoboxtope candidates. 
Hence zonoboxtope (candidates) are projections of boxtope (candidates). 
This generalizes the familiar fact that zonotopes are projections of~cubes.

\section{Boxtopes}
\label{sec3}

Given any polytope $P$, its \emph{f-vector} is the tuple $(f_0, f_1, \dots, f_d)$ where $f_i$ 
denotes the number of $i$-dimensional faces of $P$, and $d = \dim(P)$.
In this section, we identify the \emph{maximal $d$-boxtopes}.
By this, we mean those $d$-boxtopes whose f-vector is maximal in each component.

An important ingredient in the proof is Proposition \ref{prop:boxtope-cubical}, in which we  show that maximal $d$-boxtopes with $2^{d+1}$ vertices are \emph{cubical}.
A polytope $P$ is called \emph{cubical} if all faces are combinatorially equivalent to cubes.
Cubicality is a recurring theme in this paper. 
In Corollary \ref{cor: general position cubical},
 we show that a broad class of zonoboxtope candidates are cubical.
Finally in Section \ref{sec7}, we demonstrate that good architectures of arbitrary depth yield cubical polytopes.

We now introduce two specific families of cubical polytopes, which give maximal boxtopes.
The first one exists in any dimension $d$. It generalizes the polytope on the right in Figure \ref{fig:endler}:
\begin{equation}
  \label{eqn:maximal-boxtope}
  B_d \,\,:= \,\,\conv\left(\, [-1, 1]^{\lceil d/2\rceil} \times [-2, 2]^{\lfloor d/2\rfloor} \;\cup\; [-2, 2]^{\lceil d/2\rceil}\times [-1, 1]^{\lfloor d/2\rfloor}\, \right).
\end{equation}
The second family can only be constructed when the dimension $d$ is odd:
\begin{equation}
  \label{eqn:maximal-boxtope-prime} 
 \!\! B_d' \,:= \,\conv\left(\,[-1, 1]^{\lfloor d/2\rfloor}\times [-2, 1]\times [-2, 2]^{\lfloor d/2\rfloor} \;\cup\; [-2, 2]^{\lfloor d/2\rfloor}\times [-1, 2]\times [-1, 1]^{\lfloor d/2\rfloor} \, \right)\!.
\end{equation}
When $d$ is odd, $B_d$ and $B'_d$ have the same f-vector. 
For example, that common f-vector equals
$$ \hbox{$(16,28,14)$ for $d=3$,
$\,\,(32, 80, 72, 24)$ for $d=4$, and
$\,\,(64, 192, 232, 136, 34)$ for $d=5$.} 
$$
The f-vector of $B_d$  is given by the generating function in Corollary~\ref{cor:max-boxtope-f-vector}; see \cite[Section~4.2]{JZ} for f-vectors of more general cubical polytopes. 
We now state the main result of this section.

\begin{theorem} 
  \label{thm:maxbox}
  The polytopes $B_d$ in (\ref{eqn:maximal-boxtope}) and $B_d'$ in (\ref{eqn:maximal-boxtope-prime}) are maximal boxtopes.
   Every maximal $d$-boxtope with $2^{d+1}$ vertices is combinatorially equivalent to either $B_d$ or $B_d'$.
\end{theorem}

The proof that $B_d$ and $B'_d$ are realizable in the form (\ref{eqn:boxtope-def}) is deferred until the end of the section.
We instead begin with an a priori description of the set of  maximal $d$-boxtope candidates with $2^{d+1}$ vertices.
The following result shows that such polytopes are cubical.

\begin{proposition}
  \label{prop:boxtope-cubical}
  Let $B = \conv(B^{(a)} \cup B^{(b)})$ be a maximal $d$-boxtope candidate with $2^{d+1}$ vertices.
  Then $B^{(a)}$ and $B^{(b)}$ are $d$-dimensional parallelepipeds and $B$ is cubical.
\end{proposition}

\begin{proof}
  Each of $B^{(a)}$ and $B^{(b)}$ must have the maximum possible $2^d$ vertices, since $B$ has $2^{d+1}$ vertices.
  This implies that $B^{(a)}$ and $B^{(b)}$ are $d$-dimensional parallelepipeds; thus, $\dim(B) = d$.

  We are left to show that $B$ is cubical.
  For this, we look ahead to Section \ref{sec:separating-hypersurface}.
  Since $B$ maximizes the f-vector among $d$-boxtope candidates with $2^{d+1}$ vertices, Lemma \ref{cor: maximal implies general position} implies that $B^{(a)}$ and $B^{(b)}$ are in general position, meaning that a small translation of either box does not change the combinatorial type of $B$.
  According to Corollary \ref{cor: general position cubical}, the convex hull of normally equivalent cubical polytopes in general position is cubical.
  Thus, $B$ is cubical.
\end{proof}

Cubical polytopes are well-studied, as is the subclass of neighborly cubical polytopes.
A cubical $d$-polytope is {\em neighborly cubical} if its $\left( \lfloor \frac{d}{2} \rfloor - 1 \right)$-skeleton 
agrees with that of the $d$-cube.
It is known that $B_d$ and $B'_d$ give the unique combinatorial types for neighborly cubical polytopes \cite[Corollary~12]{JZ}.
Joswig and Ziegler prove in \cite[Theorem~15]{JZ} that neighborly cubical polytopes maximize the f-vector 
among all cubical polytopes. We state their result:

\begin{theorem}[Cubical Upper Bound Theorem]
  \label{thm:cubical-upper-bound}
  Among cubical $d$-polytopes with $2^{d+1}$ vertices,
  the number of $i$-dimensional faces is maximized by the neighborly cubical $d$-polytopes.
  \end{theorem}

\begin{proof}[Proof of Theorem \ref{thm:maxbox}]
By  Proposition \ref{prop:boxtope-cubical}, all maximal $d$-boxtope candidates with $2^{d+1}$ vertices are cubical.
  It follows from Theorem \ref{thm:cubical-upper-bound} that, among these polytopes, the componentwise maximal f-vector is achieved exactly by the polytopes that are combinatorially equivalent to $B_d$ or to $B'_d$.
  Thus, $B_d$ and $B_d'$ are precisely the maximal $d$-boxtope \underline{candidates}.
  
  Next, we show that the boxtope candidates $B_d$ and $B_d'$ are actually $d$-boxtopes, i.\,e.\ they have
  a representation as in~(\ref{eqn:boxtope-def}).
  We realize $B_d$ by choosing the  parameters in (\ref{eqn:boxtope-def}) as follows:
  \[
  I_i = \conv(-\bfe_i, \bfe_i);
  \quad
  a_i = 
  \begin{cases}
    1 & \text{ if } i \leq \lceil d / 2 \rceil,\\
    2 & \text{ if } i > \lceil d / 2 \rceil;
  \end{cases}
  \quad
  b_i = 
  \begin{cases}
    2 & \text{ if } i \leq \lceil d / 2 \rceil,\\
    1 & \text{ if } i > \lceil d / 2 \rceil.
  \end{cases}
  \]
  Realizing $B_d'$ is more complicated, because the line segments $[-2, 1]$ and $[-1, 2]$ are not related by dilation.
  Our choice of parameters $a_1,\ldots,a_d$ is the same as for $B_d$.
  The remaining parameters for $B_d'$ are chosen as follows.
 One checks that this gives a realization (\ref{eqn:boxtope-def}) of $B_d'$:
  \[
  I_i \,= \,
  \begin{cases}
    \conv\left(-\bfe_1 + \frac{1}{2} \bfe_{\lceil d / 2 \rceil}, \bfe_1 + \frac{1}{2} \bfe_{\lceil d / 2 \rceil}\right) & \text{ if } i = 1,\\
    \conv\left(-\frac{3}{2} \bfe_{\lceil d / 2 \rceil}, \frac{3}{2} \bfe_{\lceil d / 2 \rceil}\right) & \text{ if } i = \lceil d / 2 \rceil,\\
    \conv\left(-\bfe_d - \frac{1}{2} \bfe_{\lceil d / 2 \rceil}, \bfe_d - \frac{1}{2} \bfe_{\lceil  d / 2 \rceil}\right) & \text{ if } i = d,\\
    \conv(-\bfe_i, \bfe_i) & \text{ else};
  \end{cases} \quad
  b_i = 
  \begin{cases}
    2 & \text{ if } i < \lceil d / 2 \rceil,\\
    1 & \text{ if } i \geq \lceil d / 2 \rceil.
  \end{cases} 
  \]
    We now conclude that
$B_d$ and $B_d'$ are precisely the maximal \underline{boxtopes} of dimension $d$.
\end{proof}

\section{Spaces of polytopes}
\label{sec4}

In this section we characterize the
locus of maxout polytopes inside the ambient space of candidates.
Our main result is Theorem \ref{thm:fibers}. The proof of
Proposition \ref{prop:spaceof} elucidates that ambient space.
We focus on boxtopes, zonoboxtopes, and candidates for these~objects.
We begin with an example that shows
the distinction between boxtopes and boxtope candidates.

\begin{example}[$d=2$]  
  \label{ex:blue}
Consider the following convex hull of two unit squares:
\begin{equation}
\label{eq:notboxtope}
 P \,\, = \,\, {\rm conv} \bigl( \,[0,1]^2 ,\,[2,3]^2 \, \bigr) \quad \subset \,\,\, \, \RR^2.
 \end{equation}
We will show that the candidate $P$ is not a boxtope. If it were, then by (\ref{eqn:boxtope-def}) we could write
\begin{equation}
\label{eq:notjustcandidate} 
\begin{matrix} 
[ 0,1]^2 &  = &
s_1 \cdot {\rm conv} \bigl\{ (v_{11},v_{12})\,,\,(v_{11}\!+\!u_1,v_{12})\bigr\} \,+ \,
s_2 \cdot {\rm conv} \bigl\{ (v_{21}, v_{22})\,, \, (v_{21}, v_{22} \!+\! u_2 )  \bigr\}, \smallskip \\
[ 2,3]^2 &  = &
t_1 \cdot {\rm conv} \bigl\{ (v_{11},v_{12})\,,\,(v_{11}\!+\!u_1,v_{12}) \bigr\} \,+ \,
t_2 \cdot {\rm conv} \bigr\{ (v_{21}, v_{22})\,, \, (v_{21}, v_{22} \!+ \!u_2 ) \bigr\}.
\end{matrix}
\end{equation}
Equating lower left and upper right vertices yields a system of $8$ polynomials
in $10$ unknowns:
$$
 \begin{matrix} \langle s_1 v_{11}+s_2 v_{21}-0, \,s_1 v_{12}+s_2 v_{22}-0,\,
     s_1 (v_{11}+ u_1)+s_2v_{21}-1, \,s_1 v_{12}+s_2 (v_{22}\!+\!u_2)-1, \,\, \,\\
\,\, \,  \,  t_1 v_{11}+t_2 v_{21}-2, \,t_1 v_{12}+t_2 v_{22}-2,\,\,
     t_1 (v_{11}+ u_1)+t_2v_{21}-3, \,\,t_1 v_{12}+t_2 (v_{22}+u_2)-3\, \rangle. \\
\end{matrix}
$$
One checks, using computer algebra, that this is the unit ideal $\langle 1 \rangle$.
So, it has no solutions.
\end{example}

In what follows, we fix the number $\ell = 2$ of layers and we write $(d,n,m)$ for the type.
Recall that a zonoboxtope candidate is the convex hull
of two zonotopes with the same $n$ zones in $\RR^d$. 
Among these are zonoboxtopes, which actually come from networks
and are thus written as in (\ref{eq:notjustcandidate}).
A maxout candidate is the Minkowski sum of $m$ scaled $(d, n)$-zonoboxtope candidates,
while a maxout polytope is the Minkowski sum of $m$ scaled $(d, n)$-zonoboxtopes.

\begin{proposition} 
  \label{prop:spaceof}
The space of maxout candidates has dimension $\,(d-1)n+2m(d+n)$.
\end{proposition}

\begin{proof}[Construction and proof]
Each maxout candidate $P$ is described by the following parameters.
The $n$ zones are given by {\em direction vectors} $\mathfrak{u}_1, \ldots, \mathfrak{u}_n$ in $\RR^d$. This involves a total of $nd$ parameters.
Using precisely these $n$ zones, we now create $m$ pairs of zonotopes in $\RR^d$.

Our zonotope pairs are indexed by  $k=1,2,\ldots,m$.
The $k$th pair of zonotopes is 
\begin{equation}
    \label{eq:zonopairs} \begin{matrix}
    Z_{k1} & = & \mathfrak{v}_k \, + s_{k1}  \cdot {\rm conv}(0,\mathfrak{u}_1)
     + s_{k2} \cdot {\rm conv} (0,\mathfrak{u}_2) + \,\cdots\, + 
      s_{kn} \cdot {\rm conv}(0,\mathfrak{u}_n), \\
    Z_{k2} & = & \mathfrak{w}_k \, + t_{k1} \cdot {\rm conv} (0,\mathfrak{u}_1)
    + t_{k2} \cdot  {\rm conv} (0,\mathfrak{u}_2)
    + \,\cdots\, +  t_{kn} \cdot {\rm conv} (0,\mathfrak{u}_n). \\
\end{matrix}
\end{equation}
The $2m$ {\em shift vectors} $\mathfrak{v}_i$ and $\mathfrak{w}_i$  are elements in $\RR^d$, for a total of $2md$ parameters.
The $2m$ {\em scaling vectors} $\mathfrak{s}_k = (s_{k1},s_{k2},\ldots,s_{kn})$ and $\mathfrak{t}_k = (t_{k1},t_{k2},\ldots,t_{kn})$ are elements of $\RR_{\geq 0}^n$, for a total of $2mn$ parameters.
Collecting the direction vectors, shift vectors and scaling vectors into matrices, the $m$ pairs of zonotopes 
in (\ref{eq:zonopairs}) are parametrized by the tuple of five matrices
\begin{equation}
\label{eq:5tuple} (\,\mathfrak{u},\,\mathfrak{v},\mathfrak{w},\,\mathfrak{s},\mathfrak{t}\,) . 
\end{equation}
Adding up the three parameter counts, we obtain the number $dn + 2md + 2mn$.
This overcounts the stated dimension by $n$. We shall explain this by an action of the group $\RR_{>0}^n$.

We begin by noting that the maxout candidate with parameters (\ref{eq:5tuple}) is the polytope
\begin{equation}
\label{eq:PPolytope}
 P \,\, = \,\, 
{\rm conv}(Z_{11} \cup Z_{12}) \,+\,
{\rm conv}(Z_{21} \cup Z_{22}) \,+\,\cdots + \,
{\rm conv}(Z_{m1} \cup Z_{m2}).
\end{equation}
On an open set of points (\ref{eq:5tuple}), we can arrange the shift and scaling vectors so that for all $k$, a copy of each zone in $Z_{k1}$ and a copy of each zone in $Z_{k2}$ appears as an edge in ${\rm conv}(Z_{k1} \cup Z_{k2})$; e.g., by choosing the shift vector in such a way that the convex hull has the empty sum of the generators for one zonotope and the sum of all generators of the other zonotope as vertices.
This ensures that, up to permuting labels, all $2m$ zonotopes can be uniquely recovered from the polytope $P$.
We thus represent the maxout candidate $P$ by the list of $2m$ zonotopes $Z_{kj}$.

We now observe that the representation (\ref{eq:zonopairs}) of the zonotopes is not unique. 
For this, we assume for simplicity that the zones of our zonotopes are pairwise non-parallel.
This implies that the $n$ segment summands of each of the $2m$ zonotopes are uniquely defined up to permutation and sign.
The choice of the sign, or direction, for each segment summand results in finitely many possibilities for the shift vectors $\mathfrak{v}_k$ and $\mathfrak{w}_k$ of zonotopes $Z_{k1}$ and $Z_{k2}$.

This is not the case for the direction vectors and scaling vectors. 
The zonotopes remain unchanged if we scale these vectors with elements of $\lambda = (\lambda_1,\ldots,\lambda_n)$ 
in $\RR_{>0}^n$, namely
\begin{equation}
\label{eq:ifwescale}
\mathfrak{u}_i \mapsto \frac{1}{\lambda_i} \mathfrak{u}_i,\,\,\,
s_{ki} \mapsto \lambda_i s_{ki},\,\,\,
t_{ki} \mapsto \lambda_i t_{ki}
\qquad {\rm for} \,\, i=1,\ldots,n \,\,\,\,{\rm and} \,\,\,\, k = 1,\ldots,m.
\end{equation}
We view this as an action on the space $\RR^{dn+2md+2mn}$ with coordinates (\ref{eq:5tuple}). Thus boxtope candidates are uniquely represented, at least locally, by elements of the quotient space 
\begin{equation}
\label{eq:candidatespace} \RR^{dn+2md+2mn} / \,\RR_{>0}^n. 
\end{equation}
The dimension count is therefore correct. The quotient space (\ref{eq:candidatespace}) has the stated
dimension. We can use (\ref{eq:5tuple}), modulo the scaling by $\lambda$, 
as our coordinates for the maxout candidates.~\qedhere
\end{proof}

From now on, the quotient space (\ref{eq:candidatespace}) with its coordinates $ (\,\mathfrak{u},\,\mathfrak{v},\mathfrak{w},\,\mathfrak{s},\mathfrak{t}\,) $ will be referred to as the {\em space of maxout candidates}.
If we want to eliminate the scaling action then we pass to an affine chart $ \RR^{(d-1)n+2m(d+n)}$. For instance, we could simply set 
$t_{11} = t_{12} = \cdots = t_{1n} = 1$.

\smallskip

We now turn to maxout polytopes of type $(d, n, m)$.
We denote such a polytope by $P_f$ to distinguish it from the maxout candidate $P$.
The polytope $P_f$ is determined by $2nd + 2mn$ weights,
namely the entries of real $n \times d$ matrices $A = (a_{ij})$ and $B = (b_{ij})$ and nonnegative $m \times n$ matrices $C = (c_{ki})$ and $D = (d_{ki})$.
The rows of these four matrices are denoted by
${\bf a}_i, {\bf b}_i$ for $i=1,\ldots,n$, and by
${\bf c}_k, {\bf d}_k$ for $k=1,\ldots,m$.
The space of weights is  $\RR^{2nd} \times \RR^{2mn}_{\geq 0}$.

We know that $P_f$ is the Minkowski sum of $m$ polytopes, each of which is the convex hull of two $d$-dimensional zonotopes that use the same $n$ zones.
But which such candidates arise? The maxout polytope $P_f$ equals the Minkowski sum
(\ref{eq:PPolytope}), but now the zonotopes are
\begin{equation}
    \label{eq:zonopairs-maxout}  
   \!\! \!\! \begin{matrix}
   Z_{k1} & \!\!=\!\! & c_{k1} \cdot {\rm conv}( {\bf a}_1, {\bf b}_1 )
             \,  +  c_{k2} \cdot {\rm conv} ( {\bf a}_2, {\bf b}_2 )
   + \cdots + c_{kn} \cdot {\rm conv} ({\bf a}_n,{\bf b}_n),  \\
   Z_{k2} & \!\!=\!\! & d_{k1} \cdot  {\rm conv}({\bf a}_1, {\bf b}_1 )
                 +  d_{k2} \cdot {\rm conv} ({\bf a}_2, {\bf b}_2 )
   + \cdots + d_{kn} \cdot {\rm conv}({\bf a}_n,{\bf b}_n)
  \end{matrix} \quad
  \hbox{for} \,\,\, k \in [m].
\end{equation}
We note that $P_f$ does not uniquely determine the weights of the network. 
The action of $\RR_{>0}^n$ on the space $\RR^{2nd} \times \RR_{\geq 0}^{2mn}$ preserves the polytope $P_f$.
An element $(\lambda_1, \ldots, \lambda_n) \in \RR_{>0}^n$ acts~by
$$\bfa_i \mapsto \frac{1}{\lambda_i} \bfa_i, \,\, \bfb_i \mapsto \frac{1}{\lambda_i} \bfb_i, \,\, c_{ki} \mapsto \lambda_i c_{ki}, \,\, d_{ki} \mapsto \lambda_i d_{ki}, \quad \text{for all } i = 1, \ldots, n, \,\, k = 1, \ldots, m.$$
This action leaves $Z_{k1}$ and $Z_{k2}$ unchanged and hence it preserves $P_f$.
We thus identify the {\em parameter space for maxout polytopes} with the quotient
space $(\RR^{2nd} \times \RR_{\geq 0}^{2mn})/(\RR_{>0}^n)$.

\smallskip

The formulas in (\ref{eq:zonopairs-maxout}) specify the following map from networks to maxout candidates:
\begin{equation}
\label{eq:embedding}
\begin{matrix} \phi : & \RR^{2nd} \times \RR^{2mn}_{\geq 0} & \rightarrow & \RR^{dn+2m(d+n)}, \smallskip \\ &
(A,B,\,C,D) \,\,& \mapsto & (\,\mathfrak{u},\,\mathfrak{v},\mathfrak{w},\,\mathfrak{s},\mathfrak{t}\,) .
\end{matrix}
\end{equation}
To make $\phi$ explicit, we equate
 the zonotopes in (\ref{eq:zonopairs}) with those in (\ref{eq:zonopairs-maxout}).
Note that
$\mathfrak{u}$ is a~$d \times n$ matrix,
$\mathfrak{v}$ and $\mathfrak{w}$ are $d \times m$ matrices, and
$ \mathfrak{s}$ and $\mathfrak{t}$ are $m \times n$ matrices.
Using matrix algebra,
\begin{equation}
\label{eq:matrixalgebra}
\mathfrak{v}\,=\, (CA)^\top,\,\,
\mathfrak{w} \,=\, (DA)^\top, \,\,
\mathfrak{u}\,=\, (B-A)^\top,\,\,
\mathfrak{s} \,=\,  C,\,\,
\mathfrak{t} \,=\,  D.
\end{equation}
The image of the map $\phi$ is a semialgebraic set whose points are the various maxout polytopes.
Each maxout polytope represents a unique piecewise-linear convex function $f : \RR^d \rightarrow \RR$.

We note that $\phi$ is equivariant with respect to the action of $\RR_{>0}^n$.
To see this, let $\Lambda = {\rm diag}(\lambda_1,\ldots,\lambda_n)\,$ be the $n \times n$ diagonal 
matrix which represents an element of $\RR_{>0}^n$.
In terms of the matrix algebra, the action of $\RR_{>0}^n$ on the domain
of the map (\ref{eq:embedding})  is then given by 
$$(A, B, C, D) \,\Lambda \,\,= \,\,(\Lambda^{-1} A, \Lambda^{-1} B, C \Lambda, D \Lambda).$$
After applying the map $\phi$ to the tuple on the right, we obtain
\begin{equation}
    \label{eq:phi-equivaraint}
    \mathfrak{v}\,=\, (CA)^\top,\,\,
    \mathfrak{w} \,=\, (DA)^\top, \,\,
    \mathfrak{u}\,=\, (B-A)^\top \Lambda^{-1},\,\,
    \mathfrak{s} \,=\,  C \Lambda,\,\,
    \mathfrak{t} \,=\,  D \Lambda.
\end{equation}
This is precisely the scaling action seen in (\ref{eq:ifwescale}).
Hence, $\phi$ is equivariant, and we can define
\begin{equation}
    \label{eq:embedding-bar}
    \begin{matrix} \overline{\phi} : & (\RR^{2nd} \times \RR^{2mn}_{\geq 0})/\RR_{>0}^n & \rightarrow & \RR^{dn+2m(d+n)}/\RR_{>0}^n, \smallskip \\ &
    (A,B,\,C,D) \Lambda \,\,& \mapsto & (\,\mathfrak{u},\,\mathfrak{v},\mathfrak{w},\,\mathfrak{s},\mathfrak{t}\,)\Lambda.
    \end{matrix}
\end{equation}

Our main result  concerns the locus of maxout polytopes inside the space of candidates:

\begin{theorem} 
  \label{thm:fibers}
The fibers of $\phi$ and $\overline{\phi}$ have the expected dimension
${\rm max}(0,  d n  - 2dm)$. If this maximum is $\geq 0$ then
the image of $\phi$ is a semialgebraic set of full dimension
$2nd+2mn$. 
If $dn  - 2dm < 0$ then the Zariski closure of the image of $\phi$
is a proper algebraic subvariety.
\end{theorem}

\begin{proof}
The entries of the following $(n+d) \times 2m$ matrix are coordinates on the image space:
\begin{equation}
\label{eq:nicematrix}
 \begin{pmatrix}
\, \mathfrak{s}^\top &  \mathfrak{t}^\top\, \\
\, \mathfrak{v}\,\, &  \mathfrak{w}\,\,\,
\end{pmatrix}.
\end{equation}
The diagonal  $n \times n$ matrix 
$\Lambda$ acts by rescaling the top $n$ rows of the matrix, which does not change the rank.
Therefore, we take $\Lambda = {\rm Id}_n$, the identity matrix.
It  follows from (\ref{eq:matrixalgebra}) that
$$ \begin{pmatrix} A^\top & - {\rm Id}_d \,\end{pmatrix}  \cdot \begin{pmatrix}
\, \mathfrak{s}^\top &  \mathfrak{t}^\top\, \\
\, \mathfrak{v}\,\, &  \mathfrak{w}\,\,\,
 \end{pmatrix} \,\, = \,\,
 \begin{pmatrix}
  {\bf 0} & {\bf 0} \end{pmatrix}.
 $$
 This is the zero matrix of size   $d \times 2m$.
 The left factor is a $d \times (n+d)$ matrix
 with linearly independent rows. This implies that
 the matrix in (\ref{eq:nicematrix}) has rank at most $n$.
 
 Suppose that $2m \geq n$. This is equivalent to ${\rm max}(0,  d n  - 2dm) = 0$,
 which is the hypothesis in the second sentence.
Then the constraint that (\ref{eq:nicematrix}) has rank at most $n$ is nontrivial,
and this is the only equational constraint on the image of $\phi$.
The affine variety of $(n+d) \times 2m$ matrices of rank $\leq n$ has dimension
$(n+d) n + n(2m-n) = nd + 2mn$. This does not account for the 
$d \times n$ matrix $\mathfrak{u}$ whose $nd$ entries
are unconstrained. They do not appear in (\ref{eq:nicematrix}).
We conclude that the Zariski closure of the image of $\phi$ is the
determinantal variety defined by requiring the matrix
 (\ref{eq:nicematrix}) to have rank $\leq n$.
The dimension of this variety equals $2nd + 2mn$.

If $2m < n$ then no polynomials vanish on the image of the map $\phi$.
Since the map is semialgebraic, the dimension of the generic fiber 
equals that of the image space minus the dimension of the target space.
This difference is $dn-2dm$. The same fiber dimension holds for the map
 $\overline{\phi}$, because the passage to the quotient in
 (\ref{eq:embedding-bar}) reduces both dimensions by $n$.
\end{proof}

We illustrate our algebra
and the distinction between maxout polytopes and candidates.

\begin{example}[Small networks with bottlenecks]
\label{ex:twoonetwo}
Fix $(d,n,m) = (2,1,2)$.
We see two $2 \times 2$ matrices and three $2 \times 1 $ matrices in  (\ref{eq:phi-equivaraint}).
The spaces in (\ref{eq:embedding-bar}) have dimensions $7$ and $13$.
The image of  $\overline{\phi}$ has codimension $6$. It is given by
the variety of $3 \times 4$ matrices (\ref{eq:nicematrix}) of rank one.

A typical maxout candidate $P$ is a hexagon. Namely,
$P = {\rm conv}(Z_{11},Z_{12}) + {\rm conv}(Z_{21},Z_{22})$,
where each $Z_{ij}$ is a line segment in direction $U$.
However, the maxout polytopes have only four edges, because
$Z_{k1} = c_{k1}\cdot {\rm conv}(A,B) $  and
$Z_{k2} = d_{k2} \cdot {\rm conv}(A,B) $ implies that the two
summands of $P$ have parallel edges. This mirrors  the algebraic  fact that
(\ref{eq:nicematrix}) of rank one.

This analysis extends to $m \geq 2$.
Namely, for generic weights, every  maxout polytope of
type $(2,1,m)$ is a quadrilateral, while
maxout candidates can have up to 
$2m+2$ edges.
\end{example}

\section{Separating hypersurfaces}
\label{sec:separating-hypersurface}

Maxout polytopes arise from taking  convex hulls and Minkowski sums of polytopes.
Algebraically, this
corresponds to iterating tropical addition and tropical multiplication.
In this section we characterize the combinatorial structures that appear in the convex hull operation.
   
We fix two normally equivalent polytopes $P_1$ and $P_2$ in $\RR^d$.
By definition, the faces of $P_1$ are
in one-to-one correspondence with the faces of~$P_2$.
Our goal is to describe the faces~and normal fan of 
the convex hull  $Q = \conv(P_1 \cup P_2)$.
A face of $Q$ is {\em unmixed} if it is a face of $P_1$ or $P_2$.
It is {\em mixed} otherwise. Each diagram in
Figure~\ref{fig:endler} shows the convex hull of~two normally equivalent boxes.
The unmixed faces are blue and red.
The mixed faces are yellow.
Note that any face containing a mixed face is mixed itself, since it contains at least one vertex from each of $P_1, P_2$.
The following result makes the mixed-unmixed distinction precise.

\begin{proposition}
    \label{prop: separating fan}
    For each face $F$ of $\,Q = \conv(P_1 \cup P_2)$,
    exactly one of the following holds: 
    \begin{enumerate}[label=(\alph*)]
      \item $F$ is a face of $P_1$ and not a face of $P_2$, \vspace{-0.15cm}
      \item $F$ is a face of $P_2$ and not a face of $P_1$, \vspace{-0.15cm}
      \item $F = \conv(F_1 \cup F_2)$, where $F_1, F_2$ are corresponding faces of $P_1, P_2$, 
      and $F$ has the same dimension as $F_1$ and $F_2$,  \vspace{-0.15cm}
      \item $F = \conv(F_1 \cup F_2)$, with $F_1$ and $F_2$ corresponding, and $F$ is a prism over $F_1$ and $F_2$.
      That is, $F$ is combinatorially equivalent to the product of a line segment and $F_i$.
    \end{enumerate}
  \end{proposition}

\begin{proof}
    Let $F_i := F(P_i, \bfn)$ be the face of $P_i$ in direction $\bfn \in \RR^d$ and $h_i = \max_{\bfx \in P_i} \bfn \cdot \bfx$
    the support value.
    The cases $h_1 > h_2$ and $h_2 > h_1$ give rise to  (a) and (b).
    The condition $h_1 = h_2$ splits into cases (c) and (d), depending on whether the affine spans of $F_1$ and $F_2$ coincide.

    \begin{enumerate}
        \item[(c)] If $\aff(F_1) = \aff(F_2)$, then $F  = \conv(F_1 \cup F_2)$ has the same dimension 
        as $F_1$ and $F_2$. Here $F$ is the convex hull of two normally equivalent polytopes
        that span the same~space.

        \item[(d)] If $\aff(F_1) \neq \aff(F_2)$, then we also have  $F(Q, \bfn) = \conv(F_1 \cup F_2)$.
                The affine spans of $F_1$ and $F_2$ are different, but must be parallel, since they have the same
                normal vectors.
        This implies that the affine span of $F := F(Q, \bfn)$ has dimension $\dim F_1 + 1 = \dim F_2 + 1$.

        To show that  $F$ is a prism over $F_1,F_2$, we use induction on $m = \dim F_1$.
        If $m=0$ then $F$ is the line segment from $F_1$ to $F_2$.
        Let $m \geq 1$.
        We view $F$ as a full-dimensional polytope in its affine span $\RR^{m+1}$.
The parallel hyperplanes $\aff (F_1)$ and $ \aff (F_2)$ 
        support $F$ and cut out $F_1, F_2$ as facets. 
        Any other facet-defining hyperplane $H$ intersects both $F_1$ and $F_2$.
        Denote $R_1 := H \cap F_1$ and $ R_2 := H \cap F_2$.
        Since $\aff (R_1)$ and $\aff (R_2)$ are parallel and $F_1, F_2$ lie on the same side of $H$, the faces $R_1, R_2$ of $F_1, F_2$ are corresponding.
        Moreover, $H \cap F = \conv(R_1 \cup R_2)$, and the dimension  of $H \cap F$
        exceeds that of $R_1$ and $R_2$ by one.
               Conversely, any pair of corresponding faces $R_1, R_2$ of $F_1, F_2$ yield a face $\conv(R_1 \cup R_2)$ of $F$.
        Thus, all facets of $F$ are $F_1, F_2$ and the convex hulls of corresponding facets of $F_1, F_2$, which are prisms by induction.
        We conclude that $F$ is a prism itself. \qedhere
    \end{enumerate}
\end{proof}

Type (a), (b) faces are unmixed, and type (c), (d) faces are mixed.
The set of mixed faces is upwards closed. 
Moreover, type (c) faces are upwards closed.
This follows because the affine spans of corresponding faces are parallel in general, and equal for faces of type (c).

The normally equivalent polytopes $P_1$ and $P_2$ are in {\em general position} 
if no pair of corresponding faces has the same affine span.
That is, $P_1$ and $P_2$ are in general position if $Q$ has no type (c) faces.
Next we use the general position of $P_1$ and $P_2$ to deduce properties of $Q$.

\begin{corollary}
    \label{cor: general position cubical}
    If the normally equivalent polytopes $P_1$ and $P_2$ are cubical and in general position then  $Q = 
    \conv(P_1 \cup P_2)$ is also a cubical polytope.
\end{corollary}

\begin{proof}
By hypothesis, $Q$ has no facets of type (c).
    Each facet of $Q$ of type (a) or (b) is combinatorially a cube because $P_1, P_2$ are assumed to be cubical.
    A facet of $Q$ of type (d) is a prism over a combinatorial cube, which is again combinatorially equivalent to a cube.
\end{proof}

\begin{corollary}
  \label{cor: maximal implies general position}
  If $P_1$ and $P_2$ are full-dimensional and not in general position, there exists $\bfv \in \RR^d$ such that $\conv(P_1 \cup (P_2 + \epsilon \bfv))$ has more facets than $Q$ for all sufficiently small $\epsilon > 0$.
\end{corollary}

\begin{proof}
  Faces of type (a) and (b) in $Q$ remain faces of the same type in $\conv(P_1 \cup (P_2 + \epsilon \bfv))$ for small $\epsilon$, since $|h_2 - \max_{\bfx \in P_2} (\bfx + \epsilon \bfv) \cdot \bfn| \leq \epsilon (\bfn \cdot \bfv)$.
  Type (d) faces are also preserved, provided that $\epsilon |\bfv|$ is smaller than the distance between the corresponding faces that span the type (d) face.
  Thus, for all sufficiently small $\epsilon$, only type (c) faces may change.

  The polytopes $P_1, P_2$ are not in general position, $Q$ has a facet $F$ of type (c).
  Let $F_1, F_2$ be the facets of $P_1, P_2$ corresponding to $F$.
  If $F_1 = F_2$ then all vertices of $F$ have type (c), and so all facets adjacent to $F$ are type (c), as type (c) faces are upwards closed.
  Continuing in this way, we see that either $P_1 = P_2$ or there is a type (c) facet $F$ such that $F_1 \neq F_2$.
  
  For now assume $F_1 \neq F_2$, and take $\bfv$ to be the outer unit normal vector to $F$.
  Under perturbation, the facet $F$ spawn a type (b) facet and at least one type (d) facet.
  To see this, let $R$ be a type (a) or (b) ridge of $Q$ contained in $F$, and let $R_1, R_2$ be the corresponding ridges in $P_1, P_2$.
  We assume that $R$ has type (a), hence $R = R_1$; if only ridges of type (b) exist, then 
  we switch $P_1$ and $P_2$.
  The normal cone to $R$ consists of all convex combinations of $\bfv$ and another facet normal $\bfu$ of $P$.
  Thus, $Q^\epsilon$ has a type (d) face spanned by $R_1, R_2$.

  If $P_1 = P_2$, let $\bfw \in \RR^d$ be perpendicular to the normal ray of some facet $F$ of $P$, and consider $\conv(P_1 \cup (P_2 + \epsilon \bfw))$.
  This polytope has at least as many faces as~$Q$, since $Q$ has the minimum possible f-vector.
  Moreover, the faces of $P_1$ and $ P_2 + \epsilon \bfw$ corresponding to $F$ share the same affine span; now apply the argument above to $P_1$ and $ P_2 + \epsilon \bfw$.
\end{proof}

Now, we turn to the dual point of view. 
Let $\Ncal$ denote the common normal fan of $P_1$ and $P_2$.
Each cone in the normal fan $\Mcal$ of the polytope 
$Q = \conv(P_1 \cup P_2)$ corresponds to a face of type (a),(b),(c) or (d).
We shall see that the fan $\Mcal$ is a refinement of the fan $\Ncal$.    

\begin{definition}[Separating fan] \label{def:separating-fan}
    Let $P_1, P_2\subset \RR^d$ be normally equivalent polytopes.
    The \emph{separating fan} $\Scal$ is the polyhedral fan consisting of the normal cones of \emph{mixed} faces of $Q=\conv(P_1\cup P_2)$,
        that is, the faces of type (c) and (d) in the classification of Theorem~\ref{prop: separating fan}.
\end{definition}

Note that $\Scal$ is  a polyhedral fan: 
$\Scal$ is a subset of the normal fan of $Q = \conv(P_1\cup P_2)$, closed under taking faces,
since every face of $Q$ that contains a mixed face is mixed itself;
mixedness for faces of $\Scal$ is downwards closed as mixedness for faces of $Q$ is upwards closed.  
The following statement justifies the choice of notation introducing the \emph{separating property}.

\begin{theorem}
    \label{thm: separating fan}
    Fix the polytopes $P_1,P_2,Q$ and the fans $\Ncal,\Mcal,\Scal$ as before. 
    \begin{enumerate}[label=(\arabic*)]
    \item The fan $\Mcal$ is the coarsest refinement of $\,\Ncal$ containing $\,\Scal$ as a subset. If $P_1$ and $P_2$ are in general position, then $\Scal$ consists of the cones in $\Mcal$ whose faces are not all in~$\Ncal$.
    \item For every $d$-dimensional cone $C$ in $\,\Ncal$, the intersection $D := (\cup \,\Scal) \cap C$ is a cone in $\Scal$, and  exactly one of the following separating properties holds:
        \begin{itemize}
            \item[(i)] All extremal rays of $C$ have the same type (a) or (b), and $D = \{0\}$.
            
            \item[(ii)] The extremal rays of $C$ include rays of type (c) and at most one of the types (a)
            or (b). In this case, the cone $D$ is spanned by the type (c) rays of $C$.
            
            \item[(iii)] The extremal rays of $C$ include rays of both types (a) and (b), and potentially also of type (c).
                In this case, $D$ is a $(d-1)$-cone spanned by rays of type (d) and the type (c) rays of $C$.
                  Moreover, the hyperplane $H = {\rm span}(D)$, the linear span of $D$, separates the  type (a) rays of $C$ from the type (b)  rays of $C$.
          \end{itemize}
          The same holds for any lower-dimensional cone $C$ in $\Ncal$, with the cone $D$ having dimension ${\rm dim}(C)-1$.
        Here, ${\rm span}(D)$ is a hyperplane in the subspace ${\rm span}(C)$.
    \end{enumerate}
\end{theorem}

\begin{proof}
  Let $f_P \colon \RR^d \to \RR$ denote the support function of the polytope $P \subset \RR^d$.
  For the convex hull $Q = \conv(P_1 \cup P_2)$, we have $\epi f_Q = \epi f_{P_1} \cap \epi f_{P_2}$.  
  
  First, we prove (2) for any maximal cone $C$ in $\Ncal$.
  The statement for lower-dimensional cones of $\Ncal$ follows directly from the argument for a maximal cone whose boundary contains the lower-dimensional one.
  Both $f_{P_1}$ and $f_{P_2}$ are linear on $C$, since the normal fans of $P_1, P_2$ coincide with $\Ncal$.
  Let $H_{C, 1}$ the hyperplane in $\RR^{d + 1}$, such that the graph of $f_{P_1}$ coincides with $H_{C, 1}$ on $C$.
  Denote by $H^+_{C, 1}$ the corresponding upper halfspace with respect to the last coordinate.
  Similarly, we introduce the hyperplane $H_{C, 2}$ and halfspace $H^+_{C, 2}$ for $f_{P_2}$.
  The functions $f_{P_1}, f_{P_2}$ agree with linear functions $\ell_1, \ell_2$ on $C$.
  The cone $\epi f_{P_1}$ is the intersection of the halfspaces $H^+_{C, 1}$ for all maximal cones $C$ of $\Ncal$, while $\epi f_{P_2}$ is the intersection of all $H^+_{C, 2}$. 
  The cone $\epi f_Q$ is the intersection of $H^+_{C, 1}$ and $H^+_{C, 2}$ for all maximal cones $C$ of~$\Ncal$.
  
  The cone $C$ is the normal cone for the vertices $V_1$ of $P_1$ and $V_2$ of $P_2$.
  If $H_{C, 1} = H_{C, 2}$, then $V_1 = V_2$.
  In this case, any corresponding facets of $P_1, P_2$ incident to $V_1, V_2$ have the same affine spans.
  Hence, all rays of the cone $C$ have type (c), and $C \in \Scal$, which falls under~(ii).

  Otherwise, the intersection $H_{C, 1} \cap H_{C, 2}$ is a linear subspace of $\RR^{d + 1}$ of dimension $d - 1$, and the linear projection $\pi\!: H_{C, 1} \cap H_{C, 2} \to \RR^d$ has zero kernel, because neither of $H_{C, 1}, H_{C, 2}$ contain the subspace generated by the last coordinate.
  Thus, the image $\pi(H_{C, 1} \cap H_{C, 2})$ is a hyperplane in $\RR^d$ consisting of the points $\{\bfx \in \RR^d \,\mid\, \ell_1(\bfx) = \ell_2(\bfx) \}$.
  It separates two point sets where the ``$=$'' in the last expression is replaced by ``$<$'' or ``$>$''.
  Therefore, the intersection $\pi(H_{C, 1} \cap H_{C, 2}) \cap C$ is a cone, which we call $D$ in the statement of the theorem.
  If $D$ is $\{ 0 \}$, we are in case (i).
  If $D$ is a face of $C$, we are in case (ii).
  Finally, if $D$ contains interior points of $C$, we are in case (iii).
  In the last case the cone $D$ has dimension~$d - 1$.
  
  We now prove (1).
  Each extremal ray of $\Ncal$ is also an extremal ray of $\Mcal$, because each extremal ray of $\Ncal$ defines a facet of type (a), (b) or (c) of $Q$.
  The type of the facet depends on whether the graph of the normal ray under $f_{P_1}$ lies (a) above, (b) below or (c) coincides with the graph of the normal ray under $f_{P_2}$.
  Next, suppose $\bfn, \bfn' \in \RR^d$ lie in the interiors of different cones of $\Ncal$.
  Then $\bfn, \bfn'$ correspond to different pairs of corresponding faces of $P_1, P_2$.
  Each face of $Q$ is defined by a unique pair of corresponding faces.
  Thus $\bfn, \bfn'$ define different faces of $Q$, and hence lie in interiors of different cones of $\Mcal$.
  This proves $\Mcal$ refines~$\Ncal$.
  
  Next, we show that $\Mcal$ is the coarsest refinement of $\Ncal$ containing $\Scal$.
  The fan $\Mcal$ contains $\Scal$ as a subset since $\Scal$ consists of all cones $D$ that arise in (2).
  We show that the refinement of each maximal cone $C$ of $\Ncal$ is the coarsest possible.
  This will show $\Mcal$ is as coarse as possible.
  If $C$ falls under case (i) or (ii), then the subdivision of $C$ is trivial and thus the coarsest possible.
    If $C$ is under case (iii), then $D$ divides $C$ into two cones.
  Note that any line segment whose end points lie in the interiors of the two different cones must intersect the relative interior of $D$.
  Hence, a refinement of $C$ that has $D$ as a cell is a refinement of this subdivision by $D$.
\end{proof}

  Recall that the face fan of the polar dual of a polytope is the normal fan of the polytope.
  This allows   for the following definition, 
  which refers to the notation from Definition \ref{def:separating-fan}.
 The {\em separating hypersurface} for $P_1, P_2$ is the intersection $\mathcal{S} \cap Q^\ast$.
We next determine the topology of the separating hypersurface for the extremal boxtope $B_d$ in~(\ref{eqn:maximal-boxtope}).
If $d=3$ then $\mathcal{S} \cap Q^\ast$ is one-dimensional, shown in yellow in Figure \ref{fig:separating-hypersurface}.
We see  $\mathbb{S}^0 \times \mathbb{S}^1$ for $B_3$ and $\mathbb{S}^1$ for $B_3'$.

\begin{figure}[h]
  \centering
  \includegraphics[width=0.9\textwidth, trim={2cm 3.5cm 2cm 3.5cm},clip]{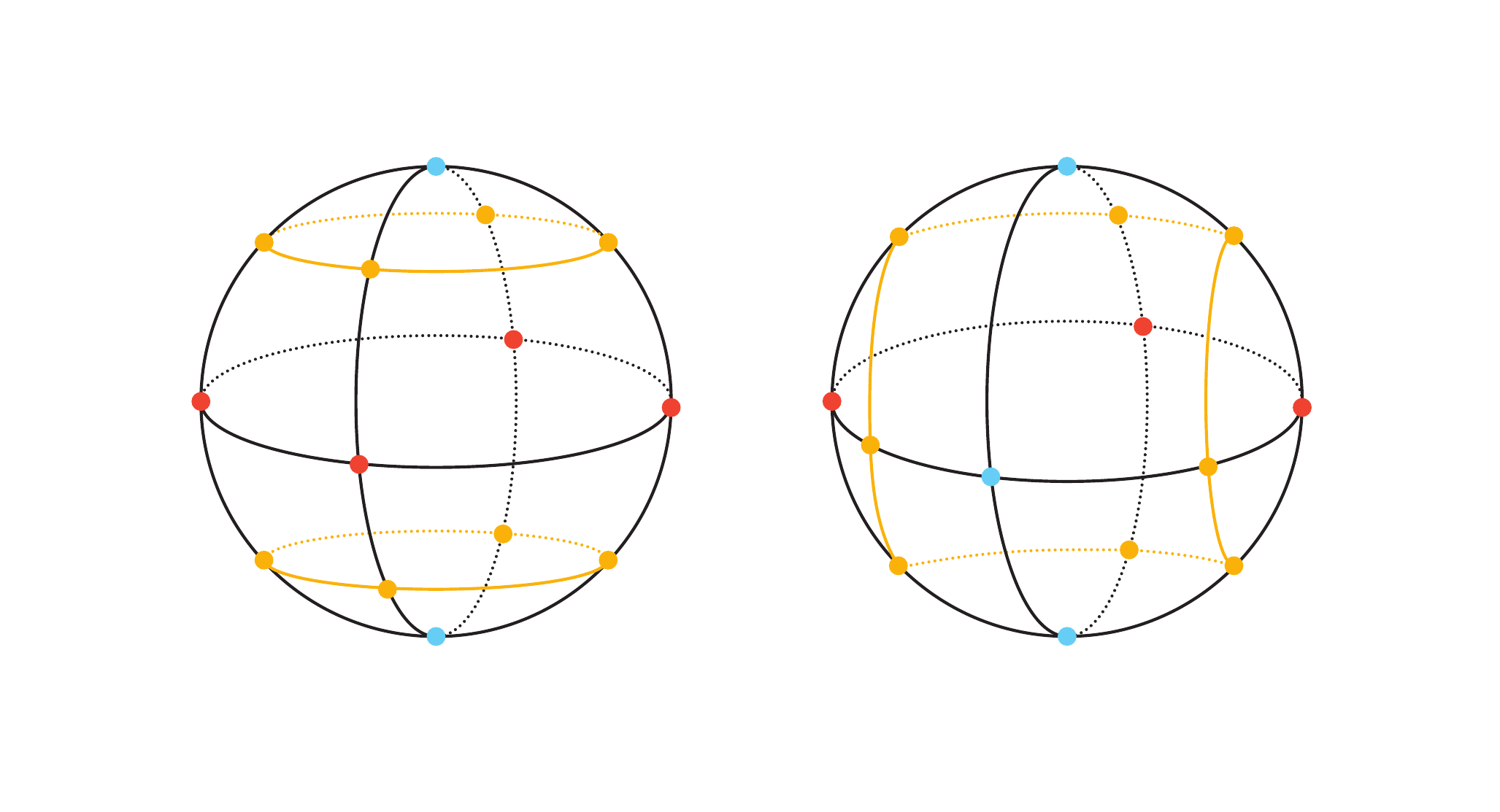}
  \caption{\label{fig:separating-hypersurface}
  Separating curves for the $3$-dimensional boxtopes $B_3$ (left) and $B_3'$ (right).}
\end{figure}

\begin{theorem} \label{thm:twocubes}
  The separating hypersurface of the extremal boxtope $B_d$  is 
  homotopy equivalent to a $(d-2)$-dimensional manifold, namely
  the product of spheres
    $\,\mathbb{S}^{\lfloor d / 2 \rfloor - 1} \times \mathbb{S}^{\lceil d / 2\rceil - 1}$.
  \end{theorem}
  
  \begin{proof}
Recall from \eqref{eqn:maximal-boxtope} that  $B_d$ is the convex hull of two $d$-cubes
    $P_1 = [-1, 1]^{\lceil d/2\rceil} \times [-2, 2]^{\lfloor d/2\rfloor}$ and $P_2 = [-2, 2]^{\lceil d/2\rceil}\times [-1, 1]^{\lfloor d/2\rfloor}$.
    Let $\Ncal$ be the normal fan of the $d$-cube, and let $\Mcal$ be the normal fan of $B_d$.
    The fan~$\Ncal$ is simplicial, its rays are generated by $\pm {\bf e}_1,\ldots,\pm {\bf e}_d$,
    and its maximal cones are the $2^d$ orthants.
    Rays of type (a) in $\Ncal$ are  spanned by $\pm {\bf e}_i$ for $i > \lceil\frac{d}{2}\rceil$, since $\max_{\bfx \in P_1} \pm \bfe_i \cdot \bfx > \max_{\bfx \in P_2} \pm \bfe_i$.
    Likewise, the rays in $\Ncal$ spanned by $\pm {\bf e}_j$ for $j \leq \lceil \frac{d}{2} \rceil$ are type (b).
    The fan $\Mcal$ has $4 \lceil \frac{d}{2} \rceil \lfloor \frac{d}{2} \rfloor$ further rays of type (d).
    These arise by applying Theorem~\ref{thm: separating fan} to every $2$-dimensional cone of~$\Ncal$ spanned by a pair of rays of type (a) and~(b). 

    Now let $C$ be a full dimensional cone of $\Ncal$.
    Theorem~\ref{thm: separating fan} tells us that $C$ is refined in $\Mcal$ if and only if $C$ has rays of type (a) and (b).
    In this case, the refining cone $D$ is spanned by the type (d) rays of $\Mcal$ contained in $C$.
    The cone $D$ is the cone over a product of two simplices, as its rays are labeled by all pairs $(i,j)$ where $i \in \{ \lceil d/2\rceil + 1, \ldots, d \}$ and $j \in \{ 1, \ldots, \lceil d/2\rceil \}$.
    
    The separating hypersurface $\mathcal{S}$ is a polyhedral complex of dimension $d-2$,
    which is contained
     in the boundary of $B_d^* = P_1^* \cap P_2^*$.
 Our description shows that $\mathcal{S}$ is the product
 of two simplicial complexes
 of dimension $\lceil d/2 \rceil -1$ and
 $\lfloor d/2 \rfloor -1$.
 More precisely, it is isomorphic to the product of the boundary of the crosspolytope whose vertices are $\pm {\bf e}_i$ for
 $i > \lceil\frac{d}{2}\rceil$, and the boundary of the crosspolytope whose vertices are $\pm {\bf e}_j$ for $j \leq \lceil \frac{d}{2} \rceil$. 
 These two factors of $\mathcal{S}$ are simplicial spheres, so $\mathcal{S}$ is homeomorphic to
     $\,\mathbb{S}^{\lfloor d / 2 \rfloor - 1} \times \mathbb{S}^{\lceil d / 2\rceil - 1}$.
\end{proof}

It is instructive to derive the face numbers $f_0,f_1,\ldots,f_{d-1},f_d$ of the boxtope $B_d$ from the
product construction above.
In what follows we write the f-vector as coefficients of polynomial
$\sum_{i=0}^{d} f_i x^i$.
We call this generating function the \emph{f-polynomial}. 

\begin{corollary}
  \label{cor:max-boxtope-f-vector}
The f-polynomial of the extremal boxtope $B_d$ equals
\begin{equation}
\label{eq:boxpolynomial}
(2+x)^d \,\, + \,\, (1+x) \cdot
\bigl(\, (2+x)^{\lceil d/2 \rceil} - x^{\lceil d/2 \,\rceil}  \bigr) \cdot
\bigl( \,(2+x)^{\lfloor d/2 \rfloor} - x^{\lfloor d/2 \,\rfloor}  \bigr).
\end{equation}
\end{corollary}

\begin{proof}
The summand $(2+x)^d$ is the f-vector of the $d$-cube.
The two factors on the right represent the f-vectors of our
two simplicial spheres.
   These are  boundaries of the two crosspolytopes,
of dimensions $\lceil d/2 \rceil $ and $\lfloor d/2 \rfloor$. 
The product gives the
f-vector of the separating hypersurface $\mathcal{S}$.
The passage from $\mathcal{N}$ to $\mathcal{M}$ creates
two new cones for each face of $\mathcal{S}$.
The new cones have consecutive dimensions.
 In  (\ref{eq:boxpolynomial})
this is accounted for by the factor $(1+x)$.
\end{proof}

The notion of separating hypersurface makes sense for any two
polytopes $P_1$ and $P_2$. They need
not be normally equivalent.
Assume that $P_1$ and $P_2$ have non-empty intersection in the interior, and that the polar dual polytopes $P_1^*$ and $P_2^*$ w.r.t.~a common interior  point intersect transversally in their boundaries.
We then define the separating hypersurface as
\begin{equation}
\label{eq:easyS}
 \mathcal{S} \,\, = \,\, \partial P_1^* \, \cap \, \partial P_2^*. 
 \end{equation}
 This is a polyhedral manifold of dimension $d-2$, embedded in the sphere $\partial( P_1^* \cap P_2^*) \simeq \mathbb{S}^{d-1}$.
 
 In our view,
the study of
(\ref{eq:easyS}) for various classes of polytopes is an interesting direction for topological combinatorics.
For instance, if $d=4$ then $\mathcal{S}$ is a closed surface. The torus $\mathbb{S}^1 \times \mathbb{S}^1$
was seen in Theorem \ref{thm:twocubes}.
One might ask for
non-orientable surfaces and surfaces of arbitrarily high genus.
How do we obtain these 
by intersecting boundaries
of $4$-polytopes?   

\section{Constructing extremal maxout polytopes}
\label{sec6}

In Section \ref{sec3}, we found that boxtopes with extremal f-vector are neighborly cubical polytopes.
This section goes beyond boxtopes.
We discuss methods for constructing maxout polytopes of type $(d, n, m)$ with the maximum possible number of vertices.
Our results are for  zonoboxtopes, which is the case $m = 1$.
This is generalized to any $m$ in Conjecture \ref{conj: zonoboxtope max vertices}.

Recall that a $(d,n)$-zonoboxtope $Q$ is a maxout polytope of type $(d, n, 1)$. We can write
\begin{equation}
  Q \,\,=\,\, \conv\left(
    \sum_{i = 1}^n a_i I_i \,\, \cup \,\, \sum_{i = 1}^n b_i I_i
  \right) \,\,\subset\,\, \RR^d,
  \label{eqn:zonoboxtope-def}
\end{equation}
where $I_i \subset \RR^d$ is a line segment, and $a_i, b_i \in \RR_{\geq 0}$ for $i \in [n]$.
Note that (\ref{eqn:zonoboxtope-def}) generalizes (\ref{eqn:boxtope-def}).

We now give a tight upper bound on the number of edges of a zonoboxtope for $d=2$.

\begin{theorem}
    \label{thm: zonoboxtopes dim two}
    The number of edges (equivalently, vertices) of a $(2, n)$-zonoboxtope  is at most $\,2n + 4 \lfloor n / 2 \rfloor$.
    This upper bound is tight for all $n \geq 2$.
\end{theorem}

Lemma \ref{lem: factoring out} gives an alternative characterization of $Q$. 
It is instrumental in the proof of Theorem \ref{thm: zonoboxtopes dim two} and in generating examples of zonoboxtopes with maximal f-vector for $d \geq 2$.

\begin{lemma}
    \label{lem: factoring out}
    Let $Q$ be a $(d, n)$-zonoboxtope. There exist $a_1^\prime, b_1^\prime, \ldots,   a_n^\prime, b_n^\prime, 
     \in \RR_{\geq 0}$, and pairwise non-parallel line segments $I_1,\ldots,I_n$ in $\RR^d$,  such that 
     $\,a_i^\prime b_i^\prime = 0$ for  $i=1,\ldots,n$, and
    \begin{equation}
        \label{eqn: factoring out}
        Q \,\,=\,\, \sum_{i = 1}^n \max(a_i^\prime, b_i^\prime) I_i \,+\, \conv\left( \sum_{i = 1}^n a_i^\prime I_i \, \cup \, \sum_{i = 1}^n b_i^\prime I_i \right) \subset \RR^d.
    \end{equation}
\end{lemma}

\begin{proof}
    Let $a_i, b_i, I_i$ be as in (\ref{eqn:zonoboxtope-def}). Then take $a_i^\prime = a_i - \min(a_i, b_i)$, and $b_i^\prime = b_i - \min(a_i, b_i)$.
\end{proof}

The equality between (\ref{eqn:zonoboxtope-def}) and 
(\ref{eqn: factoring out}) is an identity in the tropical (max-plus) semiring.
It amounts to factoring out a  monomial from a tropical polynomial.
We now prove our~theorem.

\begin{proof}[Proof of Theorem \ref{thm: zonoboxtopes dim two}]
    We begin with some notation.
    Let $Q$ be a $(2,n)$-zonoboxtope, given in factored form (\ref{eqn: factoring out}). 
    We give the following names to the Minkowski summands  in (\ref{eqn: factoring out}):
    \[ Z \,:= \,\sum_{i = 1}^n \max(a_i^\prime, b_i^\prime) I_i \quad {\rm and} \quad
    Q^\prime \,:= \,\conv\left( \sum_{i = 1}^n a_i^\prime I_i \, \cup \, \sum_{i = 1}^n b_i^\prime I_i \right) 
    \,=\, \conv\left( Z^{(a)} \, \cup \, Z^{(b)} \right). \]
    Note that $Z^{(a)}$ and $Z^{(b)}$ have no parallel generators, so they are not normally equivalent.
    
    We use the language of Section \ref{sec:separating-hypersurface} to describe the edges of
    the polygon $Q$.
    That is, an edge $e$ of $Q$ is called unmixed if it is parallel to $I_i$ for some $i$; otherwise, $e$ is called mixed.
    The number of edges of $Q$ is given by the following formula:
    \begin{equation}
        \label{eqn: zonoboxtope edge count}
        \# \text{ edges of } Q\, =\, \# \text{ edges of } Z \,+\, \# \text{ mixed edges of } Q^\prime.
    \end{equation}

    We first show that $Q$ has at most $2n + 4 \lfloor n / 2 \rfloor$ edges.
    From the condition $a_i^\prime b_i^\prime = 0$ for all $i$, it follows that if $Z^{(a)}$ has $k$ non-parallel generators, then $Z^{(b)}$ has at most $n - k$ non-parallel generators. 
    Therefore, one of $Z^{(a)}$ or $ Z^{(b)}$ has no more than $2 \lfloor n/2 \rfloor$ vertices.
    The number of mixed edges is bounded above by $2 \min(f_0(Z^{(a)}), f_0(Z^{(b)}))$.
    Using (\ref{eqn: zonoboxtope edge count}),    we conclude that
    $$\# \text{ edges of } Q\, \leq \,2n + 4 \lfloor n/2 \rfloor.$$

    We next show that, for all $n \geq 2$, there exists a $(2,n)$-zonoboxtope attaining this bound.
    We provide the data of $Z^{(a)}$ and $Z^{(b)}$, which is sufficient to reconstruct $Q$.
    For $n = 2k$, let
    \[ Z^{(a)} \,=\, \sum_{t = 0}^{k - 1} 2\cos\,\biggl(\frac{\pi}{2k}\biggr) \cdot \conv\left\{\pm e^{\frac{i \pi}{2k} t} \right\}
    \quad \text{and} \quad
     Z^{(b)} \,=\, \sum_{t = 0}^{k - 1} 2 \cos\,\biggl(\frac{\pi}{2k}\biggr) \cdot \conv\left\{\pm e^{\frac{i \pi}{2k} t + \frac{i \pi}{4k}} \right\}. \]
     Here $i = \sqrt{-1}$ and we identify $\RR^2$ with the complex plane. Similarly, for $n = 2k + 1$, let
    \[ \! Z^{(a)} \,=\, \sum_{t = 0}^{k - 1} 2\cos\biggl(\frac{\pi}{2k}\biggr) \cdot \conv\left\{\pm e^{\frac{i \pi}{2k} t} \right\}
    ~ \text{and} ~
     Z^{(b)} \,=\, \sum_{t = 0}^{k} 2 \cos\biggl(\frac{\pi}{2(k + 1)} \biggr) \cdot
     \conv\left\{\pm e^{\frac{i \pi}{2(k + 1)} t + \frac{i\pi}{2k}} \right\}. \]

    All vertices of the polygons $Z^{(a)}$ and $Z^{(b)}$ lie on the unit circle.
    Hence they are vertices of $Q'$.
    Furthermore, each vertex of $Z^{(a)}$ is adjacent to two vertices of $Z^{(b)}$ in $Q'$.
    This means that $Q'$ has $4 \lfloor n/2 \rfloor$ mixed edges, which is the maximum possible.
    Finally,  the corresponding zonotope $Z$ has $2n$ edges.
    Thus, by (\ref{eqn: zonoboxtope edge count}), $Q$ attains the upper bound of $2n + 4 \lfloor n/2 \rfloor$ edges.
\end{proof}

We now turn to maxout polytopes in arbitrary dimension $d$.
We prove an upper bound for the number of vertices by translating Section \ref{sec:separating-hypersurface} into graph bicolorings.
For any graph $G$, a {\em bicoloring} of $G$ is a map $\mathcal{C}: V(G) \to \{ a, b \}$, where $V(G)$ are the vertices of $G$.
A subgraph of $G$ is {\em monocolored} under $\mathcal{C}$ if all vertices have the same color; otherwise, it is {\em bicolored}.

Now let $P \subseteq \RR^d$ be a polytope, and let $G(P^\ast)$ denote the edge graph of
its dual $P^\ast$. 
A {\em candidate bicoloring} of~$G(P^\ast)$ is a bicoloring that arises in the following way.
Let $P_1, P_2 \subset \RR^d$ be polytopes normally equivalent to $P$ and in general position.
The candidate bicoloring~$\mathcal{C}_{P_1, P_2}$
of $G(P^\ast)$ is given by the classification of faces of $\conv(P_1 \cup P_2)$ in Proposition \ref{prop: separating fan}:
\[
\mathcal{C}_{P_1, P_2}(v) \,:=\,
\begin{cases}
  \,\,  a & \text{if the facet of $P_1$ corresponding to $v$ is also a facet of $\conv(P_1 \cup P_2)$}, \\
   \,\, b & \text{if the facet of $P_2$ corresponding to $v$ is also a facet of $\conv(P_1 \cup P_2)$}.
\end{cases}
\]
Since $P_1, P_2$ are in general position, this rule assigns a unique color to each vertex of $G(P^\ast)$.

Our goal is to bound the number of vertices of $\conv(P_1 \cup P_2)$
via candidate bicolorings of~$G(P^\ast)$. For each facet $F^\ast$ of $P^\ast$,
we write $G(F^\ast)$ for the induced subgraph on the vertices of~$P^\ast $ that lie in $F^\ast$.
Recall that $G(F^\ast)$ is bicolored  if both colors occur among these vertices.

\begin{proposition}
    \label{prop: bicoloring vertices}
    Let $P_1, P_2 \subset \RR^d$ be normally equivalent to $P \subset \RR^d$ and in general position.
    \begin{enumerate}
        \item[(1)] The number of vertices of $\conv(P_1 \cup P_2)$ is equal to the number of vertices of $P$ plus the number of bicolored facets of $G(P^\ast)$ under the candidate bicoloring~$\mathcal{C} = \mathcal{C}_{P_1, P_2}$, i.e.,
        \[ f_0(\conv(P_1 \cup P_2)) \,=\, f_0(P) \,+\, \# \{ F^\ast \text{ facet of } P^\ast : G(F^\ast) \text{ is bicolored under } \mathcal{C} \}. \]
        \item[(2)] For all facets $F^\ast$ of $P^\ast$, the following condition holds:
        \begin{equation}
            \label{eqn: valid bicoloring condition} \!\!
            \text{the induced subgraph of } G(F^\ast) \text{ on } \{ v \in F^\ast : 
            \mathcal{C}(v) = i \} \text{ is connected for } i = a, b. 
        \end{equation}
    \end{enumerate}
    \end{proposition}

\begin{proof}
    Let $\mathcal{N}$ be the normal fan to $P_i$, and let $\mathcal{M}$ be the normal fan to $\conv(P_1 \cup P_2)$.
    By Theorem \ref{thm: separating fan}, each cone of $\mathcal{N}$ is refined into at most two cones in $\mathcal{M}$.
    Thus, the number of vertices of $\conv(P_1 \cup P_2)$ is the number of vertices of $P$ plus the number of 
    maximal cones in $\mathcal{N}$ that are  refined in $\mathcal{M}$.
    The graph of a facet of $P^*$ is bicolored in $G(P^*)$ if and only if the corresponding maximal 
    cone in $\mathcal{N}$ is refined in $\mathcal{M}$, and thus part (1) follows.
  For part (2), if $G(F^\ast)$ monocolored under $\mathcal{C}$, then (\ref{eqn: valid bicoloring condition}) holds, since the graph of a polytope is connected.
    Otherwise, the cone $C \in \mathcal{N}$ corresponding to $F^\ast$ has rays of type (a) and of type (b). 
    Let $R$ be a $(d-1)$-polytope in $\RR^d$ such that $C$ is the cone over $R$.
    By Theorem \ref{thm: separating fan}, there is a hyperplane $H$ separating the type (a) and type (b) rays in $C$.
    This hyperplane cuts $R$ into two pieces. The edge subgraph of $R$ induced on each side of the split remains connected.
\end{proof}

We say that a bicoloring of $G(P^\ast)$ that satisfies (\ref{eqn: valid bicoloring condition}) for a facet $F^\ast$ of $P^\ast$ is {\em valid on} $F^\ast$.
If a bicoloring satisfies (\ref{eqn: valid bicoloring condition}) for all facets of $P^\ast$, then it is {\em valid}.
  Proposition \ref{prop: bicoloring vertices} implies:
  
\begin{corollary}
    \label{cor: bicoloring upper bound}
    With $P_1, P_2 ,P$ as above,
        the number of vertices of $\conv(P_1 \cup P_2)$ is at most  that number for $P$ plus the maximal number of bicolored facets in a valid bicoloring of~$G(P^\ast)$.
\end{corollary}

We  apply this corollary to bound the vertex numbers of  zonoboxtopes in dimension three.

\begin{proposition}
    \label{prop: upper bound experiments}
    A $(3,n)$-zonoboxtope has at most $16, 26, 44, 60$ vertices
     for $n = 3, 4, 5, 6$. These bounds are tight.
\end{proposition}

\begin{proof}
    For fixed $n$, there are finitely many combinatorial types of three-dimensional zonotopes with $n$ generators.
    For each type $Z$, we perform a depth-first search that finds a valid bicoloring of $G(Z^\ast)$ with at most $m$ monocolored facets, or confirms that no such bicoloring exists.
    By Corollary \ref{cor: bicoloring upper bound}, the smallest $m$ for which a valid bicoloring exists gives an upper bound on the maximum number of vertices of
         the convex hull of two normally equivalent zonotopes which are combinatorially equivalent to $Z$.
    Taking the maximum over all combinatorial types in with $n$ generators gives upper bound for all
    zonoboxtope candidates.

We prove that this bound is tight by exhibiting extremal zonoboxtopes.
These were found by sampling, using a strategy that     
     is inspired by Lemma \ref{lem: factoring out} and the extremal examples in Theorem \ref{thm: zonoboxtopes dim two}.
    We randomly pick line segments $I_1, \ldots, I_n$ with endpoints uniformly randomly sampled from the unit sphere in $\RR^3$, and $a_1^\prime, \ldots, a_{\lfloor n/2 \rfloor}^\prime, b_{\lfloor n/2 \rfloor + 1}^\prime, \ldots, b_{n}^\prime$ uniformly randomly sampled from $[0,1]$.
    Using only $1000 $ samples for each $n = 3, 4, 5, 6$, this method succeeds. 
    \end{proof}

Our computations reported above led us to the following conjecture on vertex numbers:

\begin{conjecture}
    \label{conj: zonoboxtope max vertices}
    \begin{enumerate}
        \item The maximal number of vertices of a $(3, n, 1)$-maxout polytope equals
        \[ 
        4 \sum_{k=0}^{2} \binom{n-1}{k} \text{ if } n \,\text{ is odd, \
        and }~ 4 \sum_{k=0}^{2} \binom{n-1}{k} - (n-2) \text{ if } n \text{ is even.}
        \]
        \item For $4 \leq d \leq n$, the maximal number of vertices of a $(d,n,1)$-maxout polytope equals
        \[ 4 \sum_{k=0}^{d-1} \binom{n-1}{k}. \]
        \item For $2 \leq d \leq n$, the maximal number of vertices of a $(d,n,m)$-maxout polytope equals
        \[ 2 \sum_{k=0}^{d-1} \binom{m-1}{k} \cdot \max 
        \bigl\{ \,f_0(P) - 2 \sum_{k=0}^{d-1} \binom{n-1}{k} : P \text{ is a } (d,n,1)\text{-maxout polytope}\, \bigr\}. \]
    \end{enumerate}
\end{conjecture}

\begin{remark}
All results and conjectures in this section are valid
not just for zonoboxtopes, but for all zonoboxtope candidates. 
The distinction from Section \ref{sec4}
does not matter here.
\end{remark}

\section{Cubical structures in generic networks}

\label{sec7}

We saw in Sections \ref{sec3} and \ref{sec:separating-hypersurface} that certain types of maxout polytopes
are generically cubical. In particular, this holds for boxtopes.
In this section, we explain this phenomenon.
We prove that networks without bottlenecks yield cubical maxout polytopes
for generic weights.

\begin{theorem}
    \label{thm:wide-layers=>cubical}
    Let $T = (d, m_1, m_2, \ldots, m_\ell)$ be such that $m_i \geq d$ for all $i \in [\ell - 1]$. Generically, maxout polytopes of type $T$ are cubical.
    Specifically, on a Zariski open subset of the parameter space of type $T$ networks, the corresponding polytope is cubical.
\end{theorem}

The intuition for this statement comes from the dual point of view: generically, separating hypersurfaces intersect transversally so their arrangements locally resemble intersections of coordinate hyperplanes.
This draws a connection to cubical complexes arising from intersections of manifolds~\cite{SchwartzZiegler:2004}.

We stay in the realm of polytopes to prove Theorem \ref{thm:wide-layers=>cubical}. 
We show that a maxout polytope with non-zero network weights is the projection of a polytope that is combinatorially equivalent to an $M$-cube, where $M = m_1 + m_2 + \cdots + m_\ell$.
This generalizes the fact that zonotopes are projections of cubes.

For a maxout polytope $P$ of type $T = (d, m_1, m_2, \ldots, m_\ell)$, we construct the corresponding combinatorial $M$-cube as the Newton polytope of the function $F^P: \RR^d \times \RR^M \to \RR$, defined below.
Recall from (\ref{eq:recursive-polytopes}) that $P$ is specified by the following network weights: two real $d \times m_1$ matrices $A_1, B_1$, two $m_i \times m_{i-1}$ non-negative real matrices $A_i, B_i$ for $i = 2, \ldots, \ell$, and a non-negative real vector $C \in \RR_{\geq 0}^{m_\ell}$.
We define $F^P$ iteratively; $F_0^P$ is the identity on~$\RR^d$.
\begin{gather}
  \label{eqn:F-def}
  F_i^P: \RR^{d} \times \RR^{m_1 + \cdots + m_i} \to \RR^{m_i}, i = 1, \ldots, \ell \nonumber \\
  (\bfx, \xi_1, \ldots, \xi_i) \mapsto \max(A_i F_{i-1}^P(\bfx, \xi_1, \ldots, \xi_{i-1}) + \xi_i, B_i F_{i-1}^P(\bfx, \xi_1, \ldots, \xi_{i-1}) - \xi_i), \\
  F^P(\bfx, \xi_1, \ldots, \xi_\ell) = C F_\ell^P(\bfx, \xi_1, \ldots, \xi_\ell). \nonumber
\end{gather}
The idea is to introduce new variables acting as bias that determine in which of the two terms of a neuron the maximum is attained. 

\begin{proposition}
  \label{prop:big-cube}
  The function $F^P$ is convex, positively homogeneous and CPWL.
  If $P$ has non-zero network weights, $\mathrm{Newt}(F^P)$ is combinatorially equivalent to an $M$-cube.
  The projection of $\,\mathrm{Newt}(F^P) \subseteq \RR^d \times \RR^M$ to the first $d$ coordinates is the maxout polytope $P$.
\end{proposition}

\begin{proof}
  We proceed by induction on $\ell$, the number of hidden layers.
  In the base case $\ell = 0$, and the function $F^P$ coincides with network function $f$, defined in (\ref{eqn:nn}).
  The corresponding maxout polytope $P = \Newt(f) = \Newt(F^P)$ is a point in $\RR^d$, which is a 0-cube.

  Now fix a network of type $T = (d, m_1, m_2, \ldots, m_\ell)$, let $T' := (d, m_1, m_2, \ldots, m_{\ell - 1})$ and let $M' := \sum_{i = 1}^{\ell - 1} m_i$.
  By induction, we know that each component of $F^P_{\ell - 1}$ is convex, positively homogeneous and CPWL.
  If $A_\ell^{(k)}, B_\ell^{(k)}$ denote the $k$th rows of $A_\ell, B_\ell$, and $(\xi_\ell)_k$ denotes the $k$th component of $\xi_\ell$, then the $k$th component of $F^P_\ell$ is $\max( A_\ell^{(k)} F_{\ell-1}^P + (\xi_\ell)_k, B_\ell^{(k)} F_{\ell-1}^P - (\xi_\ell)_k)$.
  This function is the support function~of
  \begin{equation}
    \label{eqn:conv+-}
    P_\ell^k :=
    \conv\left(
      \left(
        \Newt\left(A_\ell^{(k)} F^P_{\ell - 1}\right) + \bfe_k
      \right)
      \cup
      \left(
        \Newt\left(B_\ell^{(k)} F^P_{\ell - 1}\right) + (-\bfe_k)
      \right)
    \right) \subset \RR^{d + M},
  \end{equation}
  where $\bfe_1, \bfe_2, \ldots, \bfe_{m_\ell}$ denote the standard basis vectors for the last factor of $\RR^d \times \RR^{M'} \times \RR^{m_\ell}$.
  Hence, the non-negative sum $F^P = C F_\ell$ is the support function of the polytope $\sum_{k=1}^{m_\ell} C_k P_\ell^k$.

  To construct a combinatorical $M$-cube for non-zero weights, we consider the polytope $Q := \sum_{k = 1}^{m_\ell} P_\ell^k$ normally equivalent to $\Newt(F^P)$. 
  Our inductive hypothesis on $\ell$ implies that any positive combination of $P_{\ell - 1}^k, k \in [m_{\ell - 1}]$, in particular $Q' := \sum_{k = 1}^{m_{\ell - 1}} P_{\ell - 1}^k$ is a combinatorial $M'$-cube.
  The last $m_\ell$ coordinates of the vertices of $Q$ are $\pm 1$.
  The $2^{m_\ell}$ vectors $\sum_{k = 1}^{m_\ell} \pm \bfe_k$ are normals to all vertices of $Q$ with the corresponding sign combination in the last $m_\ell$ coordinates. 
  Thus $Q$ is the convex hull of its $2^{m_\ell}$ faces normally equivalent to $Q'$. 
  For any $k \in [m_\ell]$, in this way the vector $\sum_{s = 1}^{k - 1} \bfe_s$ selects a face of $Q$, which is a combinatorial prism over the face selected by the linear functional $\sum_{s = 1}^k \bfe_s$, because the former one is the convex hull of two normally equivalent polytopes in parallel affine subspaces.
  Thus, subsequently removing components from the vector $\sum_{k = 1}^{m_\ell} \bfe_k$, we obtain that $Q$ is a combinatorial $M$-cube.

  Finally, note that the restriction of $F^P$ to $\RR^d \times \{0\}^M \subset \RR^{d + M}$ is the network function $f$ from (\ref{eqn:nn}). 
  Thus, $P = \Newt(f)$ is the projection of $\Newt(F^P)$ to the first $d$ coordinates.
\end{proof}

To explain how cubical structures are transferred from $\Newt(F^P)$ to $P$, we label each vertex of $\Newt(F^P)$ by a word $\sigma$ of length $M$ in the letters ``\texttt{a}'' and ``\texttt{b}'', which consists of the subwords $\sigma_1, \ldots, \sigma_\ell$ of lengths $m_1, \ldots, m_\ell$.
We denote the set of all such words by $\Sigma$ and the set of length $m_i$ subwords by $\Sigma_i$.
We construct these words layerwise starting with the empty word. 
Vertices arise from Minkowski sums of polytopes of the form~(\ref{eqn:conv+-}).
The sign of the respective new coordinate arising in this construction translates either to ``\texttt{a}'' for $+1$, or to ``\texttt{b}'' for $-1$.
Two words $\sigma, \tau \in \Sigma$ are {\em adjacent} if they differ in exactly one letter.

Next, we give an explicit description of the projection of the vertex labeled by $\sigma$ to $\RR^d$, which we prove in Lemma \ref{lem:projection}.
For each $\sigma_i \in \Sigma_i, i \in [\ell]$ define $C_i^{\sigma_i}$ as the matrix of size $m_i \times m_{i - 1}$ whose $k$th row, $k \in m_i$, is $A_i^{(k)}$ if the $k$th letter of $\sigma_i$ is $\texttt{a}$, or $B_i^{(k)}$ if the $k$th letter of $\sigma_i$ is $\texttt{b}$.
Each $\sigma \in \Sigma$ defines a $m_\ell \times d$ matrix $W^\sigma$ and a point $V^\sigma$ in $\RR^d$ as follows
\begin{equation}
  \label{eqn:labelling}
  W^\sigma := C_\ell^{\sigma_\ell} \ldots C_2^{\sigma_2} C_1^{\sigma_1} \in \RR^{m_\ell \times d},
  \qquad
  V^\sigma := (C W^\sigma)^\top \in \RR^d.
\end{equation}

The set $\{(V^\sigma)^\top \mid \sigma \in \Sigma\}$ includes all vertices of $P$. 
Each vertex of $P$ corresponds to a linear piece of the network function $f$ from (\ref{eqn:nn}).
These linear pieces are determined by the choice of $A$- or $B$-rows in each $\max$ operation in~(\ref{eqn:maxout-functions}); this choice is captured by the word $\sigma$.

\begin{lemma}
  \label{lem:projection}
  The maxout polytope is cubical if all network parameters are non-zero, and for every word $\sigma$ and any $d$ words $\tau_1, \ldots, \tau_d$ adjacent to $\sigma$, $V^\sigma \! - V^{\tau_1}, \ldots, V^\sigma \! - V^{\tau_d}$ are~independent.
\end{lemma}

\begin{proof}
  Non-zero weights guarantee that $\Newt(F^P)$ is a combinatorial $M$-cube by Lemma~\ref{prop:big-cube}.
  The vertex labeling described above is compatible with the combinatorial structure of the $M$-cube in the following sense.
  Two vertices $u, v$ of $\Newt(F^P)$ are adjacent if and only if the corresponding words $\sigma^u, \sigma^v$ are adjacent, i.e. they differ in exactly one letter.
  The first $d$ coordinates of the vertex of $\Newt(F)$ labelled by $\sigma \in \Sigma$ is mapped to $V^\sigma \in \RR^d$ because taking convex hulls and Minkowski sums commute with the projection.

  The vectors $V^{\sigma^v} - V^{\tau_1}, \ldots, V^{\sigma^v} - V^{\tau_d}$ are projections of the directions of $d$ edges of $\Newt(F^P)$ incident to $v$.
  Assuming their independence guarantees that $P$ is full-dimensional and cubical.
  To see that $P$ is cubical, let $G$ be a facet of $P$.
  By Lemma \ref{prop:big-cube}, $P$ is the projection of $\Newt(F^P)$.
  Let $\tilde{G}$ be the preimage of $G$ under this projection, and let $v$ be a vertex of $\tilde{G}$.
  Assuming that the projections of the directions of any $d$ edges of $\tilde{G}$ incident to $v$ are independent implies $\tilde{G}$ is $(d-1)$-dimensional and that the projection from $\tilde{G}$ to $G$ is an affine isomorphism. 
  Thus, $P$ is full-dimensional and cubical.
\end{proof}

\begin{proof}
  [Proof of Theorem \ref{thm:wide-layers=>cubical}]

  Recall that $T = (d, m_1, \ldots, m_\ell)$ is such that $m_i \geq d$ for all $i \in [\ell - 1]$.
  For network weights satisfying the hypotheses of Lemma \ref{lem:projection}, the corresponding maxout polytope $P$ of type $T$ is cubical.
  Here, we show that there is a non-emtpy Zariski open subset of the parameter space of type $T$ maxout polytopes on which these conditions are~satisfied.

  The intersection of finitely many non-empty Zariski open subsets is non-empty, so we show that each condition is satisfied on a non-empty Zariski open subset.
  The subset of weights that are non-zero and such that $P$ is full-dimensional is a non-empty Zariski open since it is defined by polynomial inequations.
  For the latter, see Corollary \ref{cor:short}.
  It remains to show that for any $\sigma \in \Sigma$ and any $\tau_1, \ldots, \tau_d \in \Sigma$ adjacent to $\sigma$, there is a Zariski open subset of network weights on which the vectors $V^\sigma - V^{\tau_1}, \ldots, V^\sigma - V^{\tau_d}$ are linearly independent.

  Suppose $\sigma, \tau_i \in \Sigma$ differ only in the $k_i$th letter of the $j_i$th subword. Then $V^\sigma - V^{\tau_i}$ is
  \begin{equation*}
    V^\sigma - V^{\tau_i} \,\,=\,\, C C^{\sigma_\ell} \cdots C^{\sigma_{{j_i}+1}}
    \begin{pmatrix}
      0 \\
      \pm (A_{j_i}^{k_i} - B_{j_i}^{k_i}) \\
      0
    \end{pmatrix}
    C^{\sigma_{{j_i}-1}} \cdots C^{\sigma_1}.
  \end{equation*}

  On the left, $\lambda_i := C C^{\sigma_\ell} \cdots C^{\sigma_{{j_i}+1}}$ is a $1 \times m_{j_i}$ vector; on the right, $M_i := C^{\sigma_{{j_i}-1}} \cdots C^{\sigma_1}$ is an $m_{{j_i}-1} \times d$ matrix.
  Thus, $V^\sigma - V^\tau = \pm \lambda_k (A_i^k - B_i^k) M$.
  Since the network weights are non-zero, $\lambda_i \in (\RR_{>0})^{m_{j_i}}$.
  Therefore, the vectors $V^\sigma - V^{\tau_1}, \ldots, V^\sigma - V^{\tau_d}$ are linearly independent if and only if the rescaled vectors $(A_{j_1}^{k_1} - B_{j_1}^{k_1}) M_1, \ldots, (A_{j_d}^{k_d} - B_{j_d}^{k_d}) M_d$ are linearly independent.
  
  We now proceed by induction on $d$.
  In the base case, $d = 1$ and $V^\sigma - V^\tau = \lambda_k (A_i^k - B_i^k) M$, where $M \in \RR^d$ is a column vector.
  The assumption $m_i \geq d$ for all $i \in [\ell - 1]$ ensures there is a Zariski open subset on which $M_i$ has full rank.
  In this case, that means $M$ is non-zero and the equation $V^\sigma - V^\tau = 0$ is a non-zero polynomial in the entries of $A_i, B_i$.
  Therefore, there is a non-empty Zariski open subset on which $V^\sigma - V^\tau \neq 0$.

  For the general case, let $j := \min\{ j_1, \ldots, j_d \}$ and $N := C^{\sigma_{j-1}} \cdots C^{\sigma_1}$.
  For each $i \in [d]$, we have $M_i = N_i N$.
  The independence of $V^\sigma - V^{\tau_1}, \ldots, V^\sigma - V^{\tau_d}$ is equivalent to
  \begin{equation}
    \label{eqn:factored-mat}
    \det \left(
    \begin{pmatrix}
      (A_{j_1}^{k_1} - B_{j_1}^{k_1}) N_1 \\
      \vdots \\
      (A_{j_d}^{k_d} - B_{j_d}^{k_d}) N_d
    \end{pmatrix}
    N \right) 
    \neq 0.
  \end{equation}
  We regard the above determinant as a polynomial in the entries of $A_j, B_j$, 
  and we perform a row expansion along the $j$th row, which is $(A^k_j - B_j^k) N$.
  The coefficient of each term is a $(d-1) \times (d-1)$ minor of the matrix in (\ref{eqn:factored-mat}).
  Our inductive hypothesis on networks of type $T = (d-1, m_1, \ldots, m_j - 1, \ldots, m_\ell)$ ensures
  that this coefficient is non-zero.
\end{proof}

We conclude this paper by returning to the case $\ell = 2$ which was studied in Section \ref{sec4}.

\begin{corollary} 
  \label{cor:short}
    A generic maxout polytope of type $(d, n, m)$ 
    has dimension $\min(d, 2n)$, 
    and it is cubical provided $d \leq n$. 
    The dimension bound    $\leq 2n$ does not hold for
    maxout candidates.
  \end{corollary}

\begin{proof}[Proof and discussion]
The $2m$ zonotopes $Z_{ki}$
in (\ref{eq:zonopairs-maxout}) are specified by the $2n$ vectors
${\bf a}_1, {\bf b}_1,\ldots, {\bf a}_n, {\bf b}_n \in \RR^d$.
The maxout polytope $P \subset \RR^d$ is contained in the convex
cone $C$ generated by these vectors. Hence,
${\rm dim}(P) \leq 2n$. This is interesting when $2n<d$.
For $d \leq n$,
we know from Theorem \ref{thm:wide-layers=>cubical}
that the maxout polytope $P$ is cubical for generic weights.

It follows from the geometric description of
maxout candidates in Section \ref{sec4}
that these are full-dimensional
in $\RR^d$ when the weights are generic.
We thus have identified a dimensional distinction between
maxout polytopes and candidates. For instance, 
if $(d,n,m) = (3,1,2)$ then maxout polytopes are
$2$-dimensional but maxout candidates are $3$-dimensional.
\end{proof}

Theorem \ref{thm:wide-layers=>cubical} generally fails for types of
network that have bottlenecks. In what follows, we present the
smallest scenario. This is a three-dimensional generalization of
Example \ref{ex:twoonetwo}.

\begin{proposition}
  \label{prop:non-cubcial}
Maxout polytopes of type $(3,2,3)$ are never cubical 
for generic weights, and have two or four 
hexagonal faces. The latter case gives
the  extremal f-vector $(24,40,18)$.
\end{proposition}

\begin{figure}[!h]
  \centering
  \includegraphics[width=.55\textwidth, trim={3cm 3.5cm 3cm 3.5cm},clip]{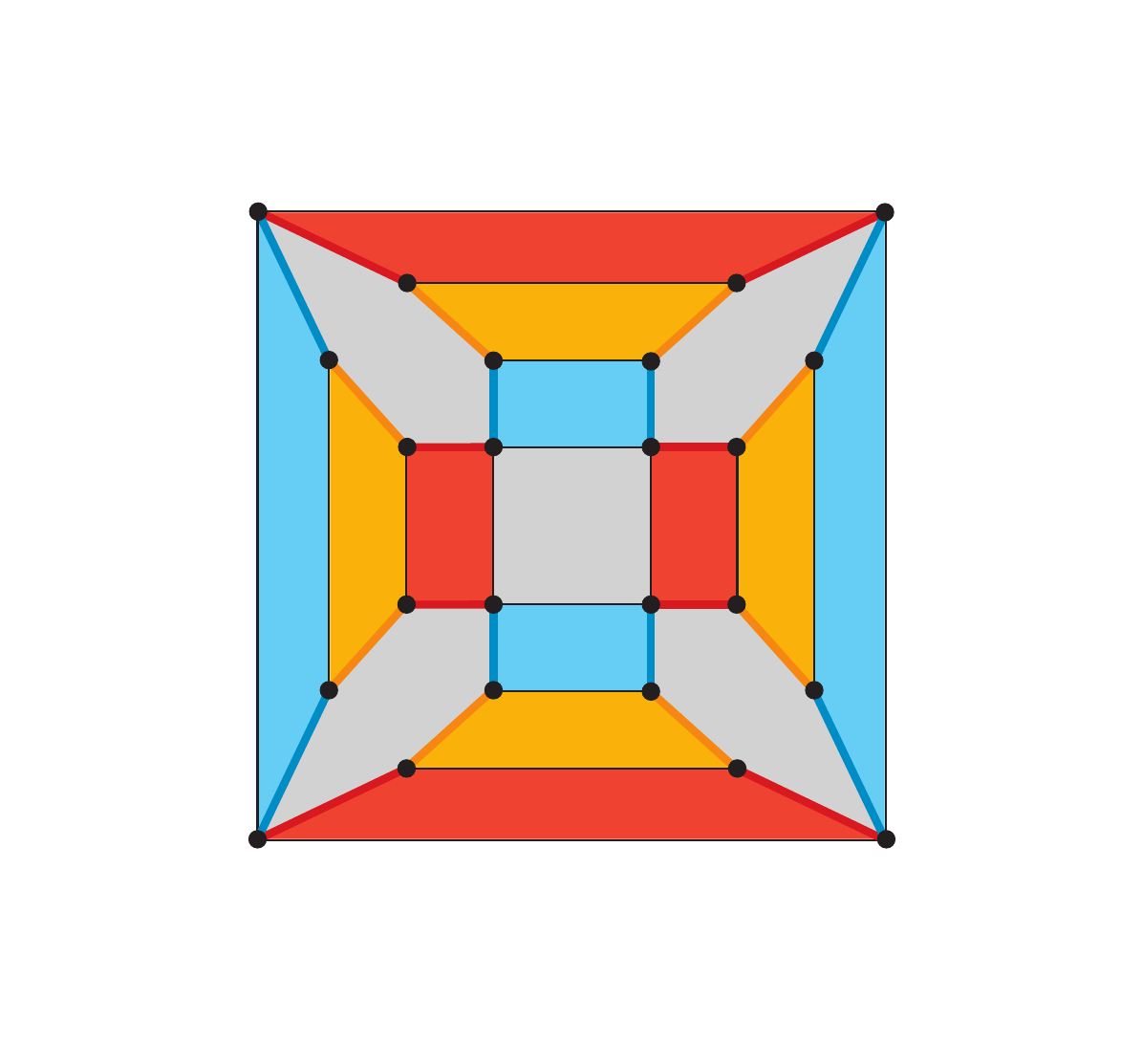}
  \caption{A generic $(3,2,3)$-maxout polytope with four hexagonal faces.}
  \label{fig:non-cubcial}
\end{figure}

\begin{proof}[Proof and Discussion]
The six parallelograms $Z_{ki}$
in (\ref{eq:zonopairs-maxout}) are generated by the vectors
${\bf a}_1, {\bf b}_1, {\bf a}_2, {\bf b}_2 \in \RR^3$.
The maxout polytope $P$ is contained in the convex
cone $C$ that is generated by these four vectors.
Either two or four of the pairs
$\{ \bfa_1, \bfa_2 \}$,
$\{ \bfa_1, \bfb_2 \}$, $\{ \bfb_1, \bfa_2 \}$ and $\{ \bfb_1, \bfb_2 \}$
must span a facet $F$ of $C$. After relabeling, let
$F = \RR_{\geq 0} \{ \bfa_1, \bfa_2 \}$ be that facet.
Each of the three Minkowski summands of $P$ intersects $F$
in an edge. Therefore, $P \cap F$ is a 
$2$-dimensional zonotope with $3$ zones.
We conclude that
 $P \cap F$ is a hexagonal facet of $P$.
 The f-vector is found by a case analysis.
Figure \ref{fig:non-cubcial} shows the
extremal maxout polytope.

It is instructive to match the geometry just seen,
i.e.~the coplanarity of our three~edges, with the algebra
in Section~\ref{sec4}.
The spaces in (\ref{eq:embedding-bar}) have dimensions $22$ and $34$.
The image of  $\overline{\phi}$ has codimension $12$. It is given by
the variety of $5 \times 6$ matrices (\ref{eq:nicematrix}) of rank $\leq 2$.
\end{proof}

\bigskip
\noindent
\footnotesize {\bf Authors' addresses:}

\noindent Andrei Balakin, Technische Universit\"at Berlin \hfill {\tt balakin@math.tu-berlin.de}

\noindent Shelby Cox, MPI-MiS Leipzig \hfill {\tt  spcox@umich.edu}

\noindent Georg Loho, Freie Universit\"at Berlin \hfill {\tt georg.loho@math.fu-berlin.de}

\noindent Bernd Sturmfels, MPI-MiS Leipzig \hfill {\tt bernd@mis.mpg.de}

\end{document}